\documentclass[12pt, a4paper,reqno]{amsart}

\allowdisplaybreaks

\usepackage[charter]{mathdesign}
\usepackage[T1]{fontenc}
\usepackage[utf8]{inputenc}

\usepackage{tipa} 

\usepackage{microtype}

\let\epsilon=\varepsilon

\let\phi=\varphi

\usepackage{enumitem}
\setlist[enumerate,1]{label=\rm(\arabic*)}
\setlist[enumerate,2]{label=\rm(\alph*)}
\setlist[enumerate,3]{label=\rm(\roman*)}

\usepackage[PS]{diagrams}
\usepackage{xcolor}
\usepackage{pgf}

\newcommand{\iso}{\cong}
\newcommand{\ie}{\textit{i.e.}}

\newcommand{\resp}{\textit{resp.}}

\newcommand{\ignore}[1]{\relax}
\newcommand{\bN}{\mathbb{N}}
\newcommand{\bZ}{\mathbb{Z}}

\newcommand{\nbd}{\nobreakdash}
\newcommand{\id}{\ensuremath{\mathrm{id}}}
\DeclareMathOperator{\tensor}{\otimes}

\numberwithin{equation}{section}

\newtheorem{theorem}[equation]{Theorem}
\newtheorem{corollary}[equation]{Corollary}
\newtheorem{proposition}[equation]{Proposition}
\newtheorem{lemma}[equation]{Lemma}

\newtheorem*{theorem*}{Theorem}

\theoremstyle{definition}
\newtheorem{definition}[equation]{Definition}

\newtheorem{example}[equation]{Example}
\newtheorem{notation}[equation]{Notation}

\let\oldtocsubsection=\tocsubsection
\renewcommand{\tocsubsection}[2]{\hspace{3.5em}\(\cdot\)~\oldtocsubsection{#1}{#2}}

\usepackage[hypertexnames=true,colorlinks=false,pdfborderstyle={/S/U/W 1.5}]{hyperref}
\usepackage{bookmark}
\usepackage{hyperref}

\newcommand{\inv}{^{-1}}

\DeclareMathOperator{\coker}{coker}

\newcommand{\RR}{R}
\newcommand{\RRR}{\RR_{*}}
\newcommand{\R}{\RRR[t,t\inv]}
\newcommand{\Rp}[1]{t^{#1}\Rpp}
\newcommand{\Rn}[1]{t^{#1}\Rnn}
\newcommand{\Rpp}{\RRR[t]}
\newcommand{\Rnn}{\RRR[t\inv]}

\newcommand{\powers}[1]{[\kern -1pt [{#1}] \kern -1pt ]} 
\newcommand{\nov}[1]{(\kern -1.7pt ( {#1})\kern-1.7pt )} 
\newcommand{\fd}{$R_{0}$\nbd-finitely dominated}

\newcommand{\Novp}{\RRR \nov {t}}         
\newcommand{\Novm}{\RRR \nov {t\inv}}     
\newcommand{\Psp}{\RRR \powers {t}}       
\newcommand{\Psm}{\RRR \powers {t\inv}}   

\DeclareMathOperator{\cone}{\mathrm{cone}}

\newcommand{\tr}{\mathrm{tr}}

\newcommand{\eqv}{\ensuremath{\Leftrightarrow}}
\newcommand{\impl}{\ensuremath{\Rightarrow}}

\newcommand{\po}{\text{\rm\raisebox{-0.7pt}{\scalebox{1.5}{\ensuremath{\cdot}}}\hglue0.8pt\llap{\raisebox{-1em}{\scalebox{2.5}{\textopencorner}}}}}

\newcommand{\htorusp}{\mathfrak{H}^{+}}

\begin{document}

\title%
[Non-commutative localisation and finite domination]%
{Non-commutative localisation and finite domination over strongly
  $\bZ$-graded rings}

\date{\today}

\author{Thomas H\"uttemann}

\address{Thomas H\"uttemann\\ Queen's University Belfast\\ School of
  Mathematics and Physics\\ Mathematical Sciences Research Centre\\
  Belfast BT7~1NN\\ Northern Ireland, UK\hfill\break ORCiD: 0000-0002-3323-4503}

\email{t.huettemann@qub.ac.uk}

\begin{abstract}
  Let $\RR = \bigoplus_{-\infty}^{\infty} \RR_{k}$ be a strongly
  $\bZ$-graded ring, and let $C^{+}$ be a chain complex of modules
  over the subring $P = \bigoplus_{0}^{\infty} \RR_{k}$. The complex
  $C^{+} \tensor_{P} \RR_{0}$ is contractible (\resp, $C^{+}$ is
  $\RR_{0}$\nbd-finitely dominated) if and only if
  $C^{+} \tensor_{P} L$ is contractible, where $L$ is a suitable
  non-commutative localisation of~$P$. We exhibit universal properties
  of these localisations, and show by example that an
  $\RR_{0}$\nbd-finitely dominated complex need not be
  $P$\nbd-homotopy finite.
\end{abstract}

\keywords{Non-commutative localisation, finite domination, type FP,
  strongly graded ring, {\scriptsize\textsc{Novikov}} homology,
  algebraic mapping torus, {\scriptsize\textsc{Mather}} trick.}

\subjclass[2010]{16W50, 16E99, 55U15, 18G35}

\maketitle

\ \vglue 0.5 cm \ 

\tableofcontents

\part*{Introduction.}

\subsection*{Finite domination}

Let $\RR_{0}$ be a unital ring, possibly non-commutative. A chain
complex $C$ of $\RR_{0}$\nbd-modules is called \textit{\fd{}} if it is
a retract up to homotopy of a bounded complex of finitely generated
free $\RR_{0}$\nbd-modules. When $C$ is bounded and consists of
projective $\RR_{0}$\nbd-modules, $C$ is \fd{} if and only if $C$ is
homotopy equivalent to a bounded complex of finitely generated
projective $\RR_{0}$\nbd-modules
\cite[Proposition~3.2~(ii)]{Ranicki-finiteness}; this is sometimes
expressed by saying that $C$ is ``of type~FP''.

\subsection*{Non-commutative localisation}

A \textit{$K$\nbd-ring} is a unit-preserving homomorphism
$K \rTo S$ of unital rings with domain~$K$. Let $\Sigma$
be a set of homomorphisms of finitely generated projective (right)
$K$\nbd-modules. The $K$\nbd-ring $f \colon K \rTo S$
is called \textit{$\Sigma$\nbd-inverting\/} if all the induced maps
\begin{displaymath}
  \sigma \tensor S \colon P \tensor_{K} S \rTo Q
  \tensor_{K} S \ , \qquad \big( \sigma \colon P \rTo Q \big) \in
  \Sigma
\end{displaymath}
are isomorphisms of $S$\nbd-modules. The \textit{non-commutative
  localisation of~$K$ with respect to~$\Sigma$} is the
$K$\nbd-ring
$\lambda_{\Sigma} \colon K \rTo \Sigma\inv K$ which is
initial in the category of $\Sigma$\nbd-inverting $K$\nbd-rings;
it exists for all~$\Sigma$ \cite[Theorem~4.1]{MR800853}.

\subsection*{Detecting contractibility and finite domination using
  non-com\-mu\-ta\-tive localisation}

Let $C^{+}$ be a bounded chain complex consisting of finitely
generated free modules over the polynomial ring~$\RR_{0}[t]$, where
$t$ is a (central) indeterminate commuting with all elements
of~$\RR_{0}$. Our starting point is the following
result obtained by \textsc{Ranicki}: \goodbreak

\begin{theorem*}
  There are sets $\tilde \Omega_{+}$ and $\Omega_{+}$ of square
  matrices with entries in~$\RR_{0}[t]$, considered as maps between
  finitely generated free $\RR_{0}[t]$\nbd-modules, such that
  \begin{itemize}
  \item[{\rm (A)}] %
    The induced complex 
    $C^{+} \tensor_{\RR_{0}[t]} \RR_{0}$ is
    contractible (tensor product via the map
    $\RR_{0}[t] \rTo \RR_{0}$, $t \mapsto 0$) if and only if the
    induced chain complex
    $C^{+} \tensor_{\RR_{0}[t]} \tilde \Omega_{+}\inv \RR_{0}[t]$ is
    contractible \cite[Proposition~10.13]{Ranicki-knot};
  \item[{\rm (B)}] $C^{+}$ is \fd{} if and only if the induced chain
    complex $C^{+} \tensor_{\RR_{0}[t]} \Omega_{+}\inv \RR_{0}[t]$ is
    contractible \cite[Proposition~10.11]{Ranicki-knot}.
  \end{itemize}
\end{theorem*}

\subsection*{Content of the paper}

In this note, \textsc{Ranicki}'s results are extended to a larger
class of rings containing polynomial rings as special examples. Let
$\RR = \bigoplus_{k \in \bZ} \RR_{k}$ be a $\bZ$\nbd-graded ring. The
polynomial ring $\RR[t]$ has a subring, denoted~$\Rpp$, consisting of
those polynomials $\sum_{k} r_{k} t^{k}$ with $r_{k} \in \RR_{k}$; up
to the ring isomorphism symbolised by $t \mapsto 1$, this is the
$\bN$\nbd-graded ring $\bigoplus_{k \geq 0} \RR_{k}$.  We will show
that {\it the results above remain valid mutatis mutandis if the
  polynomial ring~$\RR_{0}[t]$ is replaced by the $\bN$\nbd-graded
  ring~$\Rpp$ throughout, where in~{\rm(B)} we additionally demand the
  $\bZ$\nbd-graded ring~$\RR$ to be \emph{strongly} graded}. This last
conditions means that the multiplication map
$R_{k} \tensor_{R_{0}} R_{-k} \rTo R_{0}$ is surjective for all
$k \in \bZ$. It is surprising that the results rest exclusively on the
(strongly) graded structure of the underlying rings, and not on the
specific form of polynomial rings in one indeterminate.

\subsection*{Organisation of the paper}

The paper is divided into three parts, discussing $\bZ$\nbd-graded
rings and non-commutative localisation, contractible complexes, and
finite domination respectively. Independently, the material is divided
into numbered sections.

\subsection*{Conventions}

All rings are unital, ring homomorphisms preserve unity, and modules
are unital and right, unless stated otherwise.

\part{Graded rings, proto-contractions, and non-commutative localisation}

\section{Constructing new rings from a $\bZ$-graded ring}

For a (unital) ring~$\RR$ we can construct various polynomial and
power series rings using a central indeterminate~$t$; the rings
$\RR[t]$, $\RR[t\inv]$, $\RR[t,t\inv]$, $\RR \powers{t}$,
$\RR \powers{t\inv}$, $\RR \nov{t} = \RR \powers{t}[1/t]$ and
$\RR \nov{t\inv} = \RR \powers{t\inv}[1/t\inv]$ will be of
relevance. Elements of these rings can be written as formal sums
$\sum_{k} r_{k} t^{k}$, with suitable restrictions on the the number
and sign of indices of non-zero coefficients~$r_{k}$.

Suppose now that $\RR = \bigoplus_{k \in \bZ} \RR_{k}$ is equipped
with the structure of a $\bZ$\nbd-graded ring. We can then define
subrings of the rings above by requiring that for all $k \in \bZ$ the
coefficient $r_{k}$ of~$t^{k}$ lies in~$\RR_{k}$. The resulting rings
will be denoted by the symbols $\Rpp$, $\Rnn$, $\R$, $\Psp$, $\Psm$,
$\Novp$ and~$\Novm$, respectively. For example,
\begin{displaymath}
  \Novp = \bigcup_{p \geq 0} \Big\{ \sum_{k=-p}^{\infty} r_{k} t^{k}
  \,|\, \forall k \colon r_{k} \in \RR_{k} \Big\} \ .
\end{displaymath}
As a graded rings, $\R = \RR$ \textit{via} the map
symbolically described as $t \mapsto 1$. Similarly
$\Rpp = \bigoplus_{k \geq 0} \RR_{k}$ and
$\Rnn = \bigoplus_{k \leq 0} \RR_{k}$. 
We write
\begin{equation}
  \label{eq:t_n-R}
  t^{n} \Rpp = \bigoplus_{k \geq n} \RR_{k} %
  \quad \text{and} \quad %
  t^{n} \Rnn = \bigoplus_{k \leq n} \RR_{k} \ ,
\end{equation}
which are (left and right) modules over~$\Rpp$ and~$\Rnn$,
respectively; the symbol $t^{n} \Psm$ denotes the $\Psm$-module of
formal power series involving powers of~$t$ not exceeding~$n$.


\medbreak

For later use we introduce notation for truncation of formal power
series. For $-\infty \leq \ell < u \leq \infty$ we define
\begin{equation*}
  \tr_{\ell}^{u} \colon \sum_{k \in \bZ} r_{k} t^{k} \mapsto
  \sum_{k=\ell}^{u} r_{k} t^{k} \ ,
\end{equation*}
and abbreviations in the special cases $\ell = -\infty$ and
$u = \infty$,
\begin{equation}
  \label{eq:tr_one}
  \tr^{u} = \tr_{-\infty}^{u} \qquad \text{and} \qquad \tr_{\ell} =
  \tr_{\ell}^{\infty} \ .
\end{equation}
For example, the map
\begin{equation}
  \label{eq:tr0}
  \tr^{0} \colon \Rpp \rTo \RR_{0} \ , \quad \sum_{k=0}^{d} r_{k}
  t^{k} \mapsto r_{0}
\end{equation}
is the ``constant-coefficient'' ring homomorphism which is given
symbolically by $t \mapsto 0$.

\section{Strongly graded rings}
\label{sec:strongly_graded_rings}

\subsection*{Strongly graded rings and partitions of unity}

Let $\RR = \R$ be a $\bZ$\nbd-graded ring. A finite sum expressions $1
= \sum_{j} \alpha^{(n)}_{j} \beta^{(-n)}_{j}$
with $\alpha^{(n)}_{j_{n}} \in \RR_{n}$ and
$\beta^{(-n)}_{j_{n}} \in \RR_{-n}$ is called a \textit{partition of
  unity of type~\((n,-n)\)\/}.
The ring $\R$ is called \textit{strongly graded\/} (\textsc{Dade}
\cite[\S1]{GRD}) if there exists a partition of unity of
type~\((n,-n)\)
for every $n \in \bZ$; equivalently, if the multiplication map
\begin{equation}
  \label{eq:mult}
  \pi_{n} \colon \RR_{n} \tensor_{\RR_{0}} \RR_{-n} \rTo \RR_{0} \ , \quad
  x \tensor y \mapsto xy
\end{equation}
is surjective for every $n \in \bZ$.

\begin{lemma}
  \label{lem:pi_n-is-iso}
  If $\pi_{n}$ is onto, then $\pi_{n}$ is an isomorphism of
  $\RR_{0}$\nbd-$\RR_{0}$\nbd-bimodules.
\end{lemma}

\begin{proof}
  The map $\pi_{n}$ is clearly left and right $\RR_{0}$\nbd-linear. If
  $\pi_{n}$ is onto we can choose a partition of unity
  $1 = \sum_{j} \alpha^{(n)}_{j} \beta^{(-n)}_{j}$
  and define the right $\RR_{0}$\nbd-linear map
  \begin{displaymath}
    \kappa_{n} \colon \RR_{0} \rTo \RR_{n} \tensor_{\RR_{0}} \RR_{-n} \ ,
    \quad x \mapsto \sum_{j} \alpha_{j}^{(n)} \tensor \beta_{j}^{(-n)}
    x \ .
  \end{displaymath}
  Then we calculate
  \begin{displaymath}
    \kappa_{n} \pi_{n} (x \tensor y) = \sum_{j} \alpha_{j}^{(n)}
    \tensor \beta_{j}^{(-n)} xy = \sum_{j} \alpha_{j}^{(n)}
    \beta_{j}^{(-n)} x \tensor y = x \tensor y
  \end{displaymath}
  (using $\beta_{j}^{(-n)} x \in \RR_{0}$) so that $\pi_{n}$~is
  injective.
\end{proof}

\begin{lemma}
  \label{lem:more-pous}
  Let $\RR = \R$ be a $\bZ$\nbd-graded ring, and let
  $1 = \sum_{i} \alpha^{(m)}_{i} \beta^{(-m)}_{i}$ and
  $1 = \sum_{j} \bar \alpha^{(n)}_{j} \bar \beta^{(-n)}_{j}$ be %
  two 
  partitions of unity of types $(m,-m)$ and~$(n,-n)$,
  respectively. Then
    \begin{displaymath}
      1 = \sum_{i,j} \big(\alpha^{(m)}_{i} \alpha^{(n)}_{j} \big)
      \cdot \big(\beta^{(-n)}_{j} \beta^{(-m)}_{i} \big)
    \end{displaymath}
    is a partition of unity of type~$(m+n, -m-n)$.\qed
\end{lemma}

\begin{corollary}
  \label{cor:pou-strong}
  Partitions of unity of types~$(1,-1)$ and~$(-1,1)$ exist within the
  $\bZ$\nbd-graded ring~$\R$ if and only if it is strongly
  $\bZ$\nbd-graded. \qed
\end{corollary}

By direct calculation, similar to the proof of
Lemma~\ref{lem:pi_n-is-iso} above, one verifies:

\begin{lemma}
  \label{lem:mult_iso}
  Suppose that $ \RR = \R$ is a strongly $\bZ$\nbd-graded ring, and
  let $m \in \bZ$. The multiplication map
  \begin{displaymath}
    t^{-m} \Rpp \tensor_{\Rpp} \R \rTo \R \ , \quad x \tensor y
    \mapsto xy
  \end{displaymath}
  is an isomorphism of $\Rpp$-$\R$-bimodules, with inverse given by
  \begin{displaymath}
    z \mapsto \sum_{j} \alpha_{j}^{(-m)} \tensor \beta_{j}^{(m)} z
  \end{displaymath}
  for a partition of unity $1 = \sum_{j} \alpha_{j}^{(-m)}
  \beta_{j}^{(m)}$ of type~$(-m,m)$.\qed
\end{lemma}

Note that the inverse is independent from the choice of partition of
unity (since the multiplication map is). --- For later use, we
record an important categorical property of strongly $\bZ$\nbd-graded
rings:

\begin{lemma}
  \label{lem:incl_is_epi}
  Let $\RR = \R$ be a strongly $\bZ$\nbd-graded ring. The inclusion
  $\beta \colon \Rpp \rTo^{\subset} \R$ is an epimorphism in the
  category of (unital) rings.
\end{lemma}

\begin{proof}
  Let $f,g \colon \R \rTo S$ be ring homomorphisms satisfying the
  equality $f\beta = g\beta$. We need to show $f=g$. For this, let
  $x \in \RR_{k}$ be homogeneous of degree $k \in \bZ$. If $k \geq 0$
  we have $f(x) = f\beta(x) = g\beta(x) = g(x)$.  Otherwise, choose a
  partition of unity $1 = \sum_{j} \alpha_{j}^{(k)} \beta_{j}^{(-k)}$
  of type~$(k,-k)$. Then $\beta_{j}^{(-k)}$ and $\beta_{j}^{(-k)} x$
  lie in~$\Rpp$. Thus $f(\beta_{j}^{(-k)} x) = g(\beta_{j}^{(-k)} x)$,
  and we calculate \goodbreak
  \begin{multline*}
    f(x)  = g(1) \cdot f(x) \\
    = \sum_{j} g (\alpha_{j}^{(k)} \beta_{j}^{(-k)}) \cdot f(x) =
    \sum_{j} g (\alpha_{j}^{(k)}) \cdot g(\beta_{j}^{(-k)}) \cdot
    f(x) \\
    = \sum_{j} g (\alpha_{j}^{(k)}) \cdot f(\beta_{j}^{(-k)}) \cdot
    f(x)
    = \sum_{j} g (\alpha_{j}^{(k)}) \cdot f(\beta_{j}^{(-k)}x) \\
    = \sum_{j} g (\alpha_{j}^{(k)}) \cdot g(\beta_{j}^{(-k)}x) =
    \sum_{j} g (\alpha_{j}^{(k)}\beta_{j}^{(-k)}x) = g(x) \ .
  \end{multline*}
\end{proof}

\subsection*{Finiteness properties of strongly graded rings}

The homogeneous components of strongly graded rings are finitely
generated projective modules over the degree\nbd-$0$ subring.

\begin{lemma}
  \label{lem:strong_finiteness}
  Suppose that $\RR = \R$ is a $\bZ$\nbd-graded ring that admits a
  partition of unity of type~$(1,-1)$. Then for all
  $n \geq 1$,\goodbreak
  \begin{itemize}
  \item $\RR_{n}$ is finitely generated projective as a \emph{right}
    $\RR_{0}$\nbd-module;
  \item $\RR_{-n}$ is finitely generated projective as a \emph{left}
    $\RR_{0}$\nbd-module.
  \end{itemize}
  Similarly, if $\RR = \R$ admits a partition of unity of
  type~$(1,-1)$, then for all $n \geq 1$,
  \begin{itemize}
  \item $\RR_{n}$ is finitely generated projective as a \emph{left}
    $\RR_{0}$\nbd-module;
  \item $\RR_{-n}$ is finitely generated projective as a \emph{right}
    $\RR_{0}$\nbd-module.
  \end{itemize}
\end{lemma}

\begin{proof}
  Let $n \geq 1$, and let
  $1 = \sum_{j} \alpha_{j}^{(n)} \beta_{j}^{(-n)}$ be a partition of
  unity of type~$(n,-n)$ (existence is guaranteed by
  Lemma~\ref{lem:more-pous}). Define
  \begin{displaymath}
    f_{j} \colon \RR_{n} \rTo \RR_{0} \ , \quad x \mapsto
    \beta_{j}^{(-1)} x \ .
  \end{displaymath}
  The maps $f_{j}$ are right $\RR_{0}$\nbd-linear, and for all $x \in
  \RR_{n}$ we calculate
  \begin{displaymath}
    \sum_{j} \alpha_{j}^{(1)} \cdot f_{j} (x) = \sum_{j}
    \alpha_{j}^{(n)} \beta_{j}^{(-n)} x = x
  \end{displaymath}
  so that $\big( \alpha_{j}^{(1)},\, f_{j} \big)$ is a dual basis
  for~$\RR_{n}$. It follows that $\RR_{n}$ is a finitely generated
  projective right $\RR_{0}$\nbd-module by the dual basis lemma. ---
  All the remaining claims are proved in a similar manner.
\end{proof}

\begin{corollary}
  \label{cor:strong-means-fgp}
  Suppose that $\RR = \R$ is a strongly $\bZ$\nbd-graded ring.
  \begin{enumerate}
  \item For all $n \in \bZ$, the homogeneous component $\RR_{n}$
    of~$\R$ is a finitely generated projective left
    $\RR_{0}$\nbd-module and a finitely generated projective right
    $\RR_{0}$\nbd-module; in fact, $\RR_{n}$ is an invertible
    $\RR_{0}$\nbd-bimodule.
  \item If $M$ is a projective (left or right) $\R$\nbd-module, then
    $M$ is a projective (left or right) $\RR_{0}$\nbd-module (with
    module structure given by restriction of scalars). Similarly, any
    projective left or right module over $\Rpp$ or~$\Rnn$ is a
    projective $\RR_{0}$\nbd-module.
  \item There exists an isomorphism
    $\RR_{-m} \tensor_{\RR_{0}} \Rpp \iso \Rp {-m}$ of finitely
    generated projective right $\Rpp$\nbd-modules, for every
    $m \in \bZ$. Similarly, there exists an isomorphism
    $\RR_{m} \tensor_{\RR_{0}} \Rnn \iso \Rn {m}$ of finitely
    generated projective right $\Rnn$\nbd-modules.
  \item For all $m \in \bZ$, the module $\Rp {-m}$ is an invertible
    $\Rpp$\nbd-bi\-module, and hence is finitely generated projective
    as a left and right $\Rpp$\nbd-module. Similarly, $\Rn {m}$ is an
    invertible $\Rnn$\nbd-bi\-module, and hence is finitely generated
    projective as a left and right $\Rnn$\nbd-module.
  \end{enumerate}
\end{corollary}

\begin{proof}
  Statements~(1) and~(2) follow from Lemma~\ref{lem:pi_n-is-iso},
  Corollary~\ref{cor:pou-strong} and
  Lemma~\ref{lem:strong_finiteness}. To prove~(3) it is enough, in
  view of~(1), to establish the isomorphism. Let
  $1 = \sum_{j} \alpha_{j}^{(-m)} \beta_{j}^{(m)}$ be a partition of
  unity of type~$(-m,m)$. Then the multiplication map
  $\pi \colon \RR_{-m} \tensor_{\RR_{0}} \Rpp \rTo^{\iso} t^{-m}
  \Rpp$,
  sending $x \tensor y$ to~$xy$, has inverse given by
  $\rho \colon z \mapsto \sum_{j} \alpha_{j}^{(-m)} \tensor
  \beta_{j}^{(m)} z$.
  Indeed, by straightforward calculation,
  $\pi\rho(z) = \sum_{j} \alpha_{j}^{(-m)} \beta_{j}^{(m)} z = z$ and
  \begin{displaymath}
    \rho \pi (x \tensor y) = \sum_{j} \alpha_{j}^{(-m)} \tensor
  \beta_{j}^{(m)} xy = \sum_{j} \alpha_{j}^{(-m)} \beta_{j}^{(m)} x
  \tensor y = x \tensor y
  \end{displaymath}
  since $\beta_{j}^{(m)} x \in \RR_{0}$. --- The proof of~(4) is
  similar, using partitions of unity to show that $\Rp {m}$ is the
  inverse $\Rpp$\nbd-bimodule of $\Rp {-m}$.
\end{proof}

\section{Proto-null homotopies and proto-contractions}

Let $C$ and~$C'$ be chain complexes of right modules over the unital
ring~$K$, with differentials $d = d_{k} \colon C_{k} \rTo C_{k-1}$ and
$d'= d'_{k}$. A \textit{proto-contraction} of~$C$ consists of module
homomorphisms $s = s_{k} \colon C_{k} \rTo C_{k+1}$ such that
$ds+sd \colon C_{k} \rTo C_{k}$ is an automorphism of~$C_{k}$ for all
$k \in \bZ$. Let $g \colon C \rTo C'$ be a chain map; a
\textit{proto-null homotopy} consists of module homomorphisms
$t = t_{k} \colon C_{k} \rTo C'_{k+1}$ such that
$d't+td \colon C_{k} \rTo C'_{k}$ is an isomorphism for all
$k \in \bZ$.

\begin{lemma}
  \label{lem:proto}
  A chain complex~$C$ admits a proto-contraction if and only if it is
  contractible. A chain isomorphism $C \rTo C'$ admits a proto-null
  homotopy if and only if $C$ is contractible.
\end{lemma}

\begin{proof}
  For the first statement, see \textsc{Ranicki} \cite[Proof of
  Proposition~9.12]{Ranicki-knot}, the second statement is a variation
  of the theme.
\end{proof}

Given a ring homomorphism $f \colon K \rTo S$, the family of
maps~$s_{k}$ is called an \textit{$f$-proto-contraction\/} if the maps
$s_{k} \tensor \id$ form a proto-contraction of the induced complex
$f_{*}(C) = C \tensor_{K} S$. Similarly, the family of maps~$t_{k}$ is
called an \textit{$f$-proto-null homotopy of~$g$\/} if $g \tensor S$
is a chain isomorphism, and if the maps $s_{k} \tensor \id$ form a
proto-null homotopy of~$g \tensor S$.

\medbreak

We are interested in proto-contractions for the following
reason. Suppose we are given $C$ and $f$ as before, and another ring
homomorphism $g \colon S \rTo T$.  If $f_{*}(C) = C \tensor_{K} S$ is
contractible then
$(gf)_{*}(C) = C \tensor_{K} T \iso C \tensor_{K} S \tensor_{S} T$ is
contractible as well, since taking tensor product preserves
homotopies. If, however, $(gf)_{*}(C)$ is contractible it is not
guaranteed that $f_{*}(C)$ is contractible. In favourable
circumstances, a contraction of $(gf)_{*}(C)$ gives rise to a sequence
of maps $s_{k} \colon C_{k} \rTo C_{k+1}$ which can be shown, thanks
to special properties of the maps $f$ and~$g$, to be an
$f$-proto-contraction.

\section{Remarks on non-commutative localisation}

Let $K$ denote an arbitrary unital, possibly non-commutative ring.
For the reader's convenience we collect some standard facts about
non-com\-muta\-tive localisation

\begin{proposition}
  \label{prop:on_localisation_1}
  Let $\Sigma$ be a set of homomorphisms of finitely generated
  projective $K$\nbd-modules, and let $f \colon K \rTo S$ be a
  $K$\nbd-ring. Write $\lambda_{\Sigma} \colon K \rTo \Sigma\inv K$
  for the non-commutative localisation of~$K$ with respect
  to~$\Sigma$.
  \begin{enumerate}
  \item \label{item:loc-injective} If $f$ is $\Sigma$-inverting and
    injective, then $\lambda_{\Sigma}$ is injective.
  \item \label{item:loc-is-epi} The non-commutative localisation
    $\lambda_{\Sigma} \colon K \rTo \Sigma\inv K$ is an epimorphism in
    the category of unital rings.
  \item \label{item:saturated} Suppose that $\Sigma$ is the set of all
    those square matrices~$M$ with entries in~$K$ such that $f(M)$ is
    invertible over~$S$; we consider a square matrix of size~$k$ as a
    map of finitely generated free modules
    $\mu \colon K^{k} \rTo K^{k}$ so that $f(M)$ represents the
    induced map $\mu \tensor S \colon S^{k} \rTo S^{k}$. Let $A$ be a
    square matrix with entries in~$K$. Then $A \in \Sigma$ if and only
    if $\lambda_{\Sigma}(A)$ is invertible over the
    ring~$\Sigma\inv K$.
  \end{enumerate}
\end{proposition}

\begin{proof}
  \ref{item:loc-injective}~As $f$ is $\Sigma$\nbd-inverting%
  , it factors as 
  $K \rTo^{\lambda_{\Sigma}} \Sigma\inv K \rTo S$%
  . This forces
  $\lambda_{\Sigma}$ to be injective if $f$ is.

  \ref{item:loc-is-epi}~Suppose we have two ring homomorphisms
  $\alpha, \beta \colon \Sigma\inv K \rTo T$ with
  $\alpha \lambda_{\Sigma} = \beta \lambda_{\Sigma}$. This common
  composition is certainly $\Sigma$\nbd-inverting, so factorises
  \textit{uniquely} through~$\lambda_{\Sigma}$. This means precisely
  that $\alpha = \beta$, as required.

  \ref{item:saturated}~Since the map $f$ is $\Sigma$\nbd-invertible,
  it factors as
  $K \rTo^{\lambda_{\Sigma}} \Sigma\inv K \rTo^{\bar f} S$. If $A$ is
  a square matrix in~$\Sigma$ then $\lambda_{\Sigma}(A)$ is invertible
  in~$\Sigma\inv K$, by definition of non-commutative localisation. If
  the square matrix~$A$ with entries in~$K$ is such that
  $\lambda_{\Sigma} (A)$ is invertible, then
  $\bar f \lambda_{\Sigma}(A) = f(A)$ is invertible over~$S$ so that
  $A \in \Sigma$ by the specific choice of~$\Sigma$.
\end{proof}

We will have occasion to use the following construction of pushout
squares:

\begin{proposition}
  \label{prop:pushout_by_loc}
  Let $\Sigma$ be a set of homomorphisms of finitely generated
  projective $K$\nbd-modules, and let $f \colon K \rTo S$ be a ring
  homomorphism. The square in Fig.~\ref{fig:po_square} is a pushout in
  the category of unital rings,
  \begin{figure}[ht]
    \centering
    \begin{diagram}[small]
      K & \rTo^{\lambda_{\Sigma}} & \Sigma\inv K \\ %
      \dTo<{f} & \po & \dTo<{\bar f} \\ %
      S & \rTo[l>=4em]^{\lambda_{f_{*}(\Sigma)}} & f_{*}(\Sigma)\inv
      S %
    \end{diagram}
    \caption{A pushout square in the category of unital rings}
    \label{fig:po_square}
  \end{figure}
  where $f_{*}(\Sigma)$ denotes the set of induced maps
  $\sigma \tensor S \colon P \tensor_{K} S \rTo Q \tensor_{K} S$ with
  $\sigma \colon P \rTo Q$ an element of\/~$\Sigma$. The ring
  homomorphism $\bar f$ is obtained from the universal property
  of~$\lambda_{\Sigma}$ as the composition
  $\lambda_{f_{*}(\Sigma)} \circ f$ is $\Sigma$\nbd-inverting.  --- In
  other words, given ring homomorphisms $\beta \colon S \rTo T$ and
  $\alpha \colon \Sigma\inv K \rTo T$ such that
  $\alpha \circ \lambda_{\Sigma} = \beta \circ f$ there exists a
  uniquely determined ring homomorphism
  $\upsilon \colon f(\Sigma)\inv S \rTo T$ with
  $\beta = \upsilon \circ \lambda_{f(\Sigma)}$ and
  $\alpha = \upsilon \circ \bar f$, cf.~Fig.~\ref{fig:po}.
\end{proposition}

\begin{figure}[ht]
  \centering
  \begin{diagram}[small]
    K & \rTo^{\lambda_{\Sigma}} & \Sigma\inv K && \\ %
    \dTo<{f} && \dTo<{\bar f} & \rdTo(4,4)>{\alpha} & \\ %
    S & \rTo[l>=4em]^{\lambda_{f_{*}(\Sigma)}} & f_{*}(\Sigma)\inv S
    && \\
    & \rdTo(6,2)<{\beta} && \rdDashto(4,2)>{\exists! \upsilon\qquad\quad} & \\
    &&&&&& T
  \end{diagram}
  \caption{Universal property of pushout square}
  \label{fig:po}
\end{figure}

\begin{proof}
  As for notation, given any ring homomorphism $h \colon A \rTo B$ we
  let $h_{*}$ stand for the functor $\,\hbox{-}\, \tensor_{A} B$. ---
  To prove the Proposition we verify that the square has the universal
  property of a pushout, see Fig.~\ref{fig:po}. Let
  $\alpha\colon \Sigma\inv K \rTo T$ and $\beta \colon S \rTo T$ be
  ring homomorphisms such that $\alpha \lambda_{\Sigma} = \beta
  f$. Given a map $\sigma \colon P \rTo Q$ in~$\Sigma$ we know that
  \begin{displaymath}
    \beta_{*} f_{*} (\sigma) = (\beta f)_{*} (\sigma) = (\alpha
    \lambda_{\Sigma})_{*} (\sigma) = \alpha_{*}
    (\lambda_{\Sigma}){}_{*} (\sigma) \ ;
  \end{displaymath}
  as $(\lambda_{\Sigma})_{*}(\sigma)$ is invertible so
  is~$\beta_{*} f_{*} (\sigma)$. Hence the map $\beta$ is
  $f_{*}(\Sigma)$\nbd-in\-verting, and consequently factorises
  uniquely as $\beta = \upsilon \lambda_{f_{*}(\Sigma)}$, for some ring
  homomorphism $\upsilon \colon f_{*}(\Sigma)\inv S \rTo T$. From the
  chain of equalities
  \begin{displaymath}
    \upsilon \bar f \lambda_{\Sigma} = \upsilon \lambda_{f_{*}(\Sigma)} f
    = \beta f = \alpha \lambda_{\Sigma}
  \end{displaymath}
  we conclude that $\alpha = \upsilon \bar f$ since $\lambda_{\Sigma}$
  is an epimorphism by
  Proposition~\ref{prop:on_localisation_1}~\ref{item:loc-is-epi}.
\end{proof}

The following purely category-theoretic lemma will be applied, in the
proof of Proposition~\ref{prop:hom-unit}, in the context of strongly
graded rings and non-commutative localisation.

\begin{lemma}
  \label{lem:cat-thy}
  Suppose that we are given a commutative pushout square
  \begin{diagram}[small]
    A & \rTo^{\alpha} & B \\
    \dTo<{\beta} & \po & \dTo<{\delta} \\
    C & \rTo^{\gamma} & D
  \end{diagram}
  (in any category) with $\beta$ an epimorphism. Suppose further that
  there exists $\iota \colon C \rTo B$ with
  $\iota\beta = \alpha$. Then $\delta\iota = \gamma$, and
  $\delta$ is an isomorphism.
\end{lemma}

\begin{proof}
  First, since $\delta \iota \beta = \delta\alpha = \gamma\beta$, and
  since $\beta$ is an epimorphism, we have $\delta\iota=\gamma$. Next,
  by the universal property of pushouts there exists a (uniquely
  determined) morphism $\phi \colon D \rTo B$ with
  $\phi\delta = \id_{B}$ and $\phi\gamma = \iota$. The commutative
  diagram of Fig.~\ref{fig:used} can be completed along the dotted
  arrow by both $\id_{D}$ and~$\delta\phi$; by uniqueness, this means
  $\delta\phi = \id_{D}$.
\end{proof}

\begin{figure}[ht]
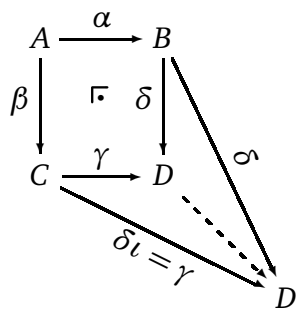

    \centering
    \begin{diagram}[small,PS]
      A & \rTo^{\alpha} & B && \\
      \dTo<{\beta} & \po & \dTo<{\delta} & \rdTo(2,4)>{\delta} & \\
      C & \rTo^{\gamma} & D && \\
      & \rdTo(4,2)<{\delta\iota = \gamma} && \rdDashto & \\
      &&&& D
    \end{diagram}
    \caption{Pushout diagram used in proof of Lemma~\ref{lem:cat-thy}}
    \label{fig:used}
  \end{figure}

\part{$\bN$-graded rings and complexes contractible over~$\RR_{0}$}

For this part we assume that $\RR = \R$ is an arbitrary $\bZ$\nbd-graded
ring; in fact, we are only interested in the subring
$\Rpp = \bigoplus_{k=0}^{\infty} \RR_{k}$ which is, in effect, an
arbitrary $\bN$\nbd-graded ring.

\section{Complexes contractible over~$\RR_{0}$}

We characterise complexes~$C$ of $\Rpp$\nbd-modules such that
$C \tensor_{\Rpp} \RR_{0}$ is contractible, where the tensor product
is taken \textit{via} the ``constant coefficient'' ring homomorphism
$\tr^{0} \colon t \mapsto 0$ of~\eqref{eq:tr0}.

\subsection*{The map~$\zeta$}

Let $M$ be an $\Rpp$\nbd-module.  Using the notation
from~\eqref{eq:t_n-R}, we write $\zeta_{M} = \zeta$ for the obvious
map of $\Rpp$-modules
\begin{equation}
  \label{eq:zeta}
  \zeta_{M} = \zeta \colon M \tensor_{\Rpp} \Rp {1} \rTo M
  \tensor_{\Rpp} \Rp {0} = M \ , \quad m \tensor x \mapsto mx
\end{equation}
induced by the inclusion map $\Rp{1} \rTo \Rp{0}$. The map~$\zeta$ is
to be thought of as a substitute for the action of the
indeterminate~$t$. More precisely, if $\Rpp = K[t]$ is a polynomial
ring, then $\Rp{1} = t K[t]$ and the composition
\begin{displaymath}
  M = M \tensor_{K[t]} K[t] \rTo^{\iso}_{\tau} M \tensor_{K[t]} \big(t K[t]
  \big) \rTo^{\zeta} M \tensor_{K[t]} K[t] = M \ ,
\end{displaymath}
where $\tau(m \tensor r) = m \tensor tr$, is given by $m \mapsto mt$;
that is, up to the isomorphism~$\tau$ the map~$\zeta$ coincides with
the action of the indeterminate.

\subsection*{Invertible matrices over~$\Psp$}

We write an element $z \in \Psp$ as a formal power series:
$z = \sum_{p \geq 0} z_{p} t^{p}$. The usual proof shows that
\textit{$z$~is a unit in~$\Psp$ if and only if $z_{0} = \tr^{0}(z)$ is
  a unit in~$\RR_{0}$}, cf.~\eqref{eq:tr_one}.

A square matrix $M$ with entries in~$\Psp$ can be written as a formal
power series $M = \sum_{p \geq 0} M_{p} t^{p}$ with matrices $M_{p}$
having entries in~$\RR_{p}$; again, the usual proof shows that
\textit{the matrix~$M$ is invertible over~$\Psp$ if and only if
  $M_{0} = \tr^{0}(M)$ is invertible over~$\RR_{0}$}.

\begin{notation}
  \label{not:tilde_Omega_plus}
  We let $\tilde \Omega_{+}$ denote the set of all square matrices~$M$
  with entries in~$\Rpp$ such that $\tr^{0} (M)$ is an invertible
  matrix over~$\RR_{0}$, that is, such that $M$ is invertible
  over~$\Psp$.
\end{notation}

We apply Proposition~\ref{prop:on_localisation_1}~\ref{item:saturated}
to the $\Rpp$\nbd-ring $f \colon \Rpp \rTo^{\subset} \Psp$:

\begin{lemma}
  \label{lem:invertible_in_tilde_loc}
  A square matrix~$M$ with entries in the $\bN$\nbd-graded ring~$\Rpp$
  becomes invertible in $\tilde \Omega_{+}\inv \Rpp$ if and only if
  $\tr^{0}(M)$ is invertible over~$\RR_{0}$.\qed
\end{lemma}

\subsection*{The localisation $\tilde\Omega_{+}\inv \Rpp$}
We consider an element $A^{+} \in \tilde\Omega_{+}$ of size~$k$ as an
endomorphism $A^{+} \colon \Rpp^{k} \rTo \Rpp^{k}$ of the finitely
generated free $\Rpp$\nbd-module $\Rpp^{k}$. The non-commutative
localisation
\begin{displaymath}
  \lambda = \lambda_{\tilde\Omega_{+}} \colon \Rpp \rTo
  \tilde\Omega_{+}\inv \Rpp
\end{displaymath}
can be used to characterise the $\RR_{0}$\nbd-contractible
complexes~$C^{+}$ as follows, generalising known results for
polynomial rings (\textsc{Ranicki}
\cite[Proposition~10.13]{Ranicki-knot}):

\begin{theorem}
  \label{thm:R0-contractibility}
  Let $\Rpp = \bigoplus_{k=0}^{\infty} \RR_{k}$ be an arbitrary
  $\bN$\nbd-graded ring, and let $C^{+}$ be a bounded complex of
  finitely generated free $\Rpp$\nbd-modules. The following statements
  are equivalent:
  \begin{enumerate}
  \item The complex $C^{+} \tensor_{\Rpp} \RR_{0}$ is
    contractible, the tensor product being taken with respect to the
    ring map $\tr^{0} \colon \Rpp \rTo \RR_{0}$, $t \mapsto 0$.
  \item The induced complex $C^{+} \tensor_{\Rpp} \Psp$ is
    contractible.
  \item The induced complex
    $C^{+} \tensor_{\Rpp} \tilde \Omega_{+}\inv \Rpp$ is contractible.
  \item The map
    $\zeta \colon C \tensor_{\Rpp} t^{1} \Rpp \rTo C \tensor_{\Rpp}
    t^{0} \Rpp$
    from~\eqref{eq:zeta} is a quasi-isomorphism.
  \end{enumerate}
\end{theorem}

\begin{proof}
  (3) \impl~(2) \impl~(1): This follows from the factorisation
  \begin{displaymath}
    \Rpp \rTo^{\subseteq} \tilde \Omega_{+}\inv \Rpp \rTo \Psp
    \rTo[l>=5em]^{\tr^{0} \colon t \mapsto 0} \RR_{0}
  \end{displaymath}
  of the ring homomorphism $\tr^{0} \colon \Rpp \rTo \RR_{0}$.

  (1) \impl~(3): We equip the finitely generated free
  modules~$C^{+}_{n}$ with arbitrary finite bases; denote the number
  of elements of the basis for~$C^{+}_{n}$ by~$r_{n}$ so that
  $C^{+}_{n}$ is identified with $\Rpp^{r_{n}}$. The differentials
  $d^{+}_{n} \colon C^{+}_{n} \rTo C^{+}_{n-1}$ are thus represented
  by matrices~$D^{+}_{n}$ of size $r_{n-1} \times r_{n}$ with entries
  in~$\Rpp$. The differentials $\tr^{0}(d^{+}_{n})$ in the induced
  complex $C^{+} \tensor_{\Rpp} \RR_{0}$ are then represented by the
  matrices~$\tr^{0}(D^{+}_{n})$, identifying
  $C^{+}_{n} \tensor_{\Rpp} \RR_{0}$ with $\RR_{0}^{r_{n}}$. By
  hypothesis there exists a contracting homotopy consisting of a
  family of $\RR_{0}$-linear maps
  \begin{displaymath} 
    \sigma^{+}_{n} \colon C^{+}_{n} \tensor_{\Rpp} \RR_{0} \rTo
    C^{+}_{n+1} \tensor_{\Rpp} \RR_{0}
  \end{displaymath}
  such that
  \begin{displaymath}
     \tr^{0}(d^{+}_{n+1}) \circ \sigma^{+}_{n} + \sigma^{+}_{n-1}
     \circ \tr^{0}(d^{+}_{n}) = \id \ .
  \end{displaymath}
  The map $\sigma^{+}_{n}$ is represented by a matrix~$S^{+}_{n}$ of
  size $r_{n+1} \times r_{n}$ with entries in~$\RR_{0}$. The matrices
  satisfy the relation
  \begin{displaymath}
    \tr^{0} \Big( D^{+}_{n+1} \circ S^{+}_{n} + S^{+}_{n-1} \circ
    D^{+}_{n} \Big) = \tr^{0} \big( D^{+}_{n+1}  \big) \circ S^{+}_{n}
    + S^{+}_{n-1} \circ \tr^{0} \big( D^{+}_{n} \big) = I_{r_{n}} \ ,
  \end{displaymath}
  a unit matrix of size~$r_{n}$. This implies, by
  Lemma~\ref{lem:invertible_in_tilde_loc}, that the matrix
  \begin{displaymath}
    D^{+}_{n+1} \circ S^{+}_{n} + S^{+}_{n-1} \circ D^{+}_{n}
  \end{displaymath}
  becomes invertible over $\tilde \Omega_{+}\inv \Rpp$. Thus
  the~$S_{n}^{+}$ define a
  $\lambda_{\tilde\Omega_{+}}$-proto-con\-trac\-tion of~$C^{+}$. With
  Lemma~\ref{lem:proto} we conclude that
  $C^{+} \tensor_{\Rpp} \tilde \Omega_{+}\inv \Rpp$ is contractible as
  advertised.

  (1) \eqv~(4): From the short exact sequence
  \begin{multline*}
    0 \rTo C^{+} \tensor_{\Rpp} \Rp {1} \rTo^{\zeta} C^{+}
    \tensor_{\Rpp} \Rp {0} \\ \rTo C^{+} \tensor_{\Rpp} \RR_{0} \rTo 0
  \end{multline*}
  we infer that the canonical map
  $\cone(\zeta) \rTo C^{+} \tensor_{\Rpp} \RR_{0}$ is a
  quasi-iso\-mor\-phism. Thus $\zeta$ is a quasi-isomorphism if and
  only if $\cone(\zeta)$ is acyclic if and only if
  $C^{+} \tensor_{\Rpp} \RR_{0}$ is acyclic; as the latter complex
  consists of projective $\RR_{0}$\nbd-modules, this is equivalent
  with $C^{+} \tensor_{\Rpp} \RR_{0}$ being contractible.
\end{proof}

\begin{theorem}[Universal property of~$\tilde \Omega_{+}\inv \Rpp$]
  \label{thm:universal-tilde}
  Let $\Rpp$ be an arbitrary $\bN$\nbd-graded ring. The localisation
  $\lambda \colon \Rpp \rTo \tilde \Omega_{+}\inv \Rpp$ is the
  universal $\Rpp$\nbd-ring making $\RR_{0}$\nbd-contractible chain
  complexes contractible. That is, suppose that $f \colon \Rpp \rTo S$
  is an $\Rpp$\nbd-ring such that for every bounded complex of
  finitely generated free $\Rpp$\nbd-modules~$C^{+}$, contractibility
  of $C^{+} \tensor_{\Rpp} \RR_{0}$ implies contractibility of
  $C^{+} \tensor_{\Rpp} S$. Then there is a factorisation
  \begin{displaymath}
    \Rpp \rTo^{\lambda} \tilde \Omega_{+}\inv \Rpp \rTo^{\eta} S
  \end{displaymath}
  of~$f$, with a uniquely determined ring homomorphism~$\eta$.
\end{theorem}

\begin{proof}
  It was shown in Theorem~\ref{thm:R0-contractibility} that the
  $\Rpp$\nbd-ring $\tilde \Omega_{+}\inv \Rpp$ makes
  $\RR_{0}$\nbd-contractible chain complexes contractible. Thus it
  is enough to verify that $f$ is $\tilde \Omega_{+}$\nbd-inverting;
  the universal property of non-commutative localisation then yields
  the desired factorisation and its uniqueness. Consider the element
  $A^{+} \in \tilde \Omega_{+}$ as a chain complex
  \begin{displaymath}
    C^{+} = \big( \Rpp^{k} \rTo^{A^{+}} \Rpp^{k} \big) \ .
  \end{displaymath}
  As $A^{+}$ becomes invertible over~$\tilde \Omega_{+}\inv \Rpp$, the
  complex $C^{+} \tensor_{\Rpp} \tilde \Omega_{+}\inv \Rpp$ is
  contractible, hence so is $C^{+} \tensor_{\Rpp} \RR_{0}$ by
  Theorem~\ref{thm:R0-contractibility}. This makes
  $C^{+} \tensor_{\Rpp} S$ contractible, by hypothesis on~$f$, whence
  $A^{+}$ becomes invertible in~$S$ as required.
\end{proof}

\part{Strongly graded rings and finite domination}

We now turn to the theory of $\RR_{0}$\nbd-finite domination of
$\Rpp$\nbd-module complexes. We characterise finite domination
\textit{via} \textsc{Novikov} homology (Theorem~\ref{thm:R0-fd-R+})
and \textit{via} a non-commutative localisation of~$\Rpp$
(Theorem~\ref{thm:non-com-fd+}). \textit{We assume throughout that
  $\RR = \R$ is a strongly $\bZ$\nbd-graded ring.}

\section{Algebraic half-tori and the \textsc{Mather} trick}

\subsection*{Algebraic half-tori and the \textsc{Mather} trick}

Let $1 = \sum_{j} \alpha^{(1)}_{j} \beta^{(-1)}_{j}$ be a partition of
unity of type $(1,-1)$ in $\RR = \R$. Given an arbitrary
$\Rpp$-module~$M$, let $\mu = \mu_{M}$ denote the map
\begin{equation}
  \label{eq:mu}
  \mu \colon M \tensor_{\RR_{0}} \Rp 1 \rTo M \tensor_{\RR_{0}} \Rp 0
  \ , \quad m \tensor x \mapsto \sum_{j} m \alpha^{(1)}_{j} \tensor \beta^{(-1)}_{j}
  x \ .
\end{equation}
The map $\mu$ is $\RR_{0}$\nbd-balanced (hence well-defined) and
independent of the choice of partition of unity since it can be
written as the composition
\begin{multline*} 
  M \tensor_{\RR_{0}} \Rp 1 \iso M \tensor_{\RR_{0}} \RR_{0}
  \tensor_{\RR_{0}} \Rp 1 \\ \iso M \tensor_{\RR_{0}} R_{1}
  \tensor_{\RR_{0}} R_{-1} \tensor_{\RR_{0}} \Rp 1
  \rTo M \tensor_{\RR_{0}} \Rp 0
\end{multline*}
where the second isomorphism is induced by
$\pi_{1}\inv \colon \RR_{0} \rTo^{\iso} \RR_{1} \tensor_{\RR_{0}}
\RR_{-1}$,
cf.~Lemma~\ref{lem:pi_n-is-iso}, and the last arrow is induced by
the multiplication maps $M \tensor_{\RR_{0}} \RR_{1} \rTo M$ and
$\RR_{-1} \tensor_{\RR_{0}} \Rp 1 \rTo \Rp 0$.

As a matter of notation, we also introduce the inclusion map
\begin{displaymath}
  \iota \colon M \tensor_{\RR_{0}} \Rp 1 \rTo M \tensor_{\RR_{0}} \Rp 0
  \ , \quad m \tensor x \mapsto m \tensor x \ .
\end{displaymath}
Moreover, it is convenient at this point to choose, once and for all,
additional partitions of unity
\begin{displaymath}
  1 = \sum_{j_{n}} \alpha^{(n)}_{j_{n}} \beta^{(-n)}_{j_{n}}
\end{displaymath}
of type~$(n,-n)$, for all $n \geq 0$ ($n \neq 1$). These exist in view
of our standing assumption for this part, that the ring $\RR = \R$ is
strongly graded.

\begin{lemma}[Canonical resolution]
  \label{lem:hal-torus-sequence}
  Suppose that $\RR = \R$ is a strongly $\bZ$\nbd-graded ring. Let $M$
  be a right $\Rpp$\nbd-module. There is a short exact sequence of
  $\Rpp$\nbd-modules
  \begin{equation}
    \label{eq:ses}
    0 \rTo M \tensor_{\RR_{0}} \Rp 1 \rTo[l>=3em]^{\iota - \mu} M
    \tensor_{\RR_{0}} \Rp 0 \rTo^{\pi} M \rTo 0 \ ,
  \end{equation}
  where $\mu$ is as in~\eqref{eq:mu},
  $\iota (m \tensor x) = m \tensor x$ and $\pi(m \tensor x) = mx$.
\end{lemma}

\begin{proof}
  This is similar to the proof of Proposition~3.2 in~\cite{MR3704930}.
  Since $\sum_{j} \alpha^{(1)}_{j} \beta^{(-1)}_{j} = 1$ we have
  $\pi \iota = \pi \mu$ and hence $\pi (\iota - \mu) = 0$. It is thus
  enough to show that the sequence is split exact when considered as a
  sequence of $\RR_{0}$\nbd-modules.

  To begin with, the map $\sigma(m) = m \tensor 1$ is certainly an
  $\RR_{0}$-linear section of~$\pi$. Next, we define the
  $\RR_{0}$\nbd-linear map
  \begin{displaymath}
    \rho \colon  M \tensor_{\RR_{0}} \Rp 0 \rTo M \tensor_{\RR_{0}}
    \Rp 1
  \end{displaymath}
  on elements of the form $m \tensor x_{d}$, with $x_{d} \in \RR_{d}$,
  by the formula
  \begin{displaymath}
    \rho (m \tensor x_{d}) = %
    \sum_{k=0}^{d-1} \sum_{j_{k}} m \alpha_{j_{k}}^{(k)} \tensor
    \beta_{j_{k}}^{(-k)} x_{d} \ .
  \end{displaymath}
  We note the particular cases
  \begin{align*}
    \rho (m \tensor x_{0}) & = 0 \ , \\
    \rho (m \tensor x_{1}) & = m \tensor x_{1} \ , \\
    \rho (m \tensor x_{2}) & = m \tensor x_{2} + \sum_{j} m \alpha^{(1)}_{j}
                             \tensor \beta^{(-1)}_{j} x_{2} \ .
  \end{align*}
  \textit{The summands
    $s_{k} = \sum_{j_{k}} m \alpha_{j_{k}}^{(k)} \tensor
    \beta_{j_{k}}^{(-k)} x_{d}$,
    and hence the map~$\rho$, do not depend on the particular choice
    of partition of unity.} This is because $s_{k}$ is the image of $m
  \tensor x_d$ under the the composition
  \begin{multline*}
    M \tensor_{\RR_{0}} \RR_{d} \iso M \tensor_{\RR_{0}} \RR_{0}
    \tensor_{\RR_{0}} \RR_{d} \\
    \rTo^{\pi_{k}\inv}_{\iso} M \tensor_{\RR_{0}} \RR_{k}
    \tensor_{\RR_{0}} \RR_{-k} \tensor_{\RR_{0}} \RR_{d} \rTo^{\sigma}
    M \tensor_{\RR_{0}} \RR_{-k+d} \ ,
  \end{multline*}
  where $\sigma(m \tensor a \tensor b \tensor x) = ma \tensor bx$, and
  $\pi_{k}\inv$ does not depend on choices by
  Lemma~\ref{lem:pi_n-is-iso}.

  We have $\rho \circ (\iota - \mu) = \id$ since, for an element
  $x_{d} \in \RR_{d}$, we calculate
  \begin{multline*}
    \rho \circ (\iota - \mu) (m \tensor x_{d}) = %
    \rho (m \tensor x_{d}) - \sum_{j} \rho (m \alpha^{(1)}_{j} \tensor
    \beta^{(-1)}_{j} x_{d}) \\ %
    = \sum_{k=0}^{d-1} \sum_{j_{k}} m \alpha_{j_{k}}^{(k)} \tensor
    \beta_{j_{k}}^{(-k)} x_{d} %
    - \sum_{k=0}^{d-2} \sum_{j_{k}} \sum_{j} m \alpha^{(1)}_{j}
    \alpha_{j_{k}}^{(k)} \tensor \beta_{j_{k}}^{(-k)} \beta^{(-1)}_{j} x_{d}
    \\ %
    \underset{(\diamond)}{=} \sum_{k=0}^{d-1} \sum_{j_{k}} m
    \alpha_{j_{k}}^{(k)} \tensor \beta_{j_{k}}^{(-k)} x_{d} %
    - \sum_{k=0}^{d-2} \sum_{j_{k+1}} m \alpha_{j_{k+1}}^{(k+1)} \tensor
    \beta_{j_{k+1}}^{(-k-1)} x_{d} \\ %
    = \sum_{j_{0}} m \alpha_{j_{0}}^{(0)} \tensor
    \beta_{j_{0}}^{(0)} x_{d} = \sum_{j_{0}} m \alpha_{j_{0}}^{(0)}
    \beta_{j_{0}}^{(0)} \tensor x_{d} = m \tensor x_{d} \ ;
  \end{multline*}
  the equality labelled~$(\diamond)$ makes use of
  Lemma~\ref{lem:more-pous}, and of the fact that summands of the
  form~$s_{k+1}$ do not depend on choice of the partition of unity
  involved. --- It remains to verify the equality
  $\sigma \circ \pi + (\iota - \mu) \circ \rho = \id$. For this, let
  $x \in \RR_{d}$ and $m \in M$, and calculate:
  \begin{multline*}
    (\iota - \mu) \circ \rho (m \tensor x_{d}) = %
    (\iota - \mu) \Big( \sum_{k=0}^{d-1} \sum_{j_{k}} m
    \alpha_{j_{k}}^{(k)} \tensor \beta_{j_{k}}^{(-k)} x_{d} \Big) \\ %
    =  \sum_{k=0}^{d-1} \sum_{j_{k}} m \alpha_{j_{k}}^{(k)} \tensor
    \beta_{j_{k}}^{(-k)} x_{d} %
    - \sum_{j} \sum_{k=0}^{d-1} \sum_{j_{k}} m \alpha_{j_{k}}^{(k)}
    \alpha^{(1)}_{j} \tensor \beta^{(-1)}_{j} \beta_{j_{k}}^{(-k)} x_{d} \\ %
    =  \sum_{k=0}^{d-1} \sum_{j_{k}} m \alpha_{j_{k}}^{(k)} \tensor
    \beta_{j_{k}}^{(-k)} x_{d} %
    - \sum_{k=0}^{d-1} \sum_{j_{k+1}} m \alpha_{j_{k+1}}^{(k+1)}
    \tensor \beta_{j_{k+1}}^{(-k-1)} x_{d} \\ %
    =  \sum_{j_{0}} m \alpha_{j_{0}}^{(0)} \tensor
    \beta_{j_{0}}^{(-0)} x_{d} %
    - \sum_{j_{d}} m \alpha_{j_{d}}^{(d)}
    \tensor \beta_{j_{d}}^{(-d)} x_{d} \\ %
    \qquad\qquad =  \sum_{j_{0}} m \alpha_{j_{0}}^{(0)} \beta_{j_{0}}^{(-0)}
    \tensor x_{d} %
    - \sum_{j_{d}} m \alpha_{j_{d}}^{(d)} \beta_{j_{d}}^{(-d)} x_{d}
    \tensor  1\\ %
    = m \tensor x_{d} - mx_{d} \tensor 1 = (\id - \sigma \circ \pi) (m
    \tensor x_{d}) \ . \qquad
  \end{multline*}
  This finishes the proof.
\end{proof}

\begin{definition}
  \label{def:half-torus}
  Let $C^{+}$ be a complex of $\Rpp$\nbd-modules.  The mapping cone
  $\htorusp (C^{+})$ of the map $\iota - \mu$,
  \begin{displaymath}
    \htorusp (C^{+}) = \cone \Big( C^{+} \tensor_{\RR_{0}} \Rp 1
    \rTo[l>=3em]^{\iota - \mu} C^{+} \tensor_{\RR_{0}} \Rp 0 \Big) \ ,
  \end{displaymath}
  is called the \textit{algebraic half-torus of~$C^{+}$}.
\end{definition}

\begin{corollary}
  \label{cor:C+-is-half-torus}
  Let $C^{+}$ be a complex of $\Rpp$\nbd-modules. The canonical map
  \begin{displaymath}
    \htorusp (C^{+}) = \cone \big( C^{+} \tensor_{\RR_{0}} \Rp 1
    \rTo[l>=3em]^{\iota - \mu} C^{+} \tensor_{\RR_{0}} \Rp 0 \big)
    \rTo C^{+}
  \end{displaymath}
  induced by the short exact sequence~\eqref{eq:ses} is a
  quasi-isomorphism. If $C^{+}$ is bounded below and consists of
  projective $\Rpp$\nbd-modules, the map is a homotopy equivalence of
  $\Rpp$\nbd-module complexes.
\end{corollary}

\begin{proof}
  This is a direct consequence of standard homological algebra and
  Lemma~\ref{lem:hal-torus-sequence} above.
\end{proof}

The following result, though technical, is central to the theory of
finite domination. By the previous Corollary we can replace any
complex~$C^{+}$ of $\Rpp$\nbd-modules by an algebraic half-torus, up
to quasi-isomorphism; the \textsc{Mather} trick is the observation
that we can further replace the complex~$C^{+}$ within the mapping
cone of the half-torus construction by an $\RR_{0}$\nbd-module complex
homotopy equivalent to~$C^{+}$.

\begin{proposition}[The algebraic \textsc{Mather} trick for algebraic half-tori]
  \label{prop:Mather_trick}
  Let $\RR = \R$ be a strongly $\bZ$\nbd-graded ring, let $C^{+}$ be a
  complex of $\Rpp$\nbd-modules, and let $D$ be a complex
  of~$\RR_{0}$\nbd-modules. Let $\alpha \colon C^{+} \rTo D$ and
  $\beta \colon D \rTo C^{+}$ be mutually inverse chain homotopy
  equivalences of $\RR_{0}$\nbd-module complexes with
  $H \colon \id \simeq \alpha \beta$ a specified homotopy.
  Write~$\psi$ for the $\Rpp$\nbd-module complex map
  \begin{displaymath}
    \psi = (\alpha \tensor \id) \circ (\iota - \mu) \circ (\beta
    \tensor \id) \colon \quad D \tensor_{\RR_{0}} \Rp 1
    \rTo D \tensor_{\RR_{0}} \Rp 0 \ .
  \end{displaymath}
  Then the
  square diagram~\eqref{eq:diag} in Fig.~\ref{fig:Mather} commutes up
  to a preferred homotopy~$J$ induced by~$H$, given by the formula
  \begin{displaymath}
    J = (\alpha \tensor \id) \circ (\iota - \mu) \circ (H \tensor \id)
    \colon (\alpha \tensor \id) \circ (\iota - \mu) \simeq \psi \circ
    (\alpha \tensor \id) \ .
  \end{displaymath}
  The homotopy~$J$ induces a preferred chain map
  \begin{displaymath}
    \Xi \colon \htorusp(C^{+}) = \cone (\iota - \mu) \rTo \cone(\psi)
  \end{displaymath}
  which is a quasi-isomorphism. If both $C^{+}$ and~$D$ are bounded
  below complexes of projective $\RR_{0}$\nbd-modules, the map
  $\htorusp(C^{+}) \rTo \cone(\psi)$ is a homotopy equivalence of
  $\Rpp$\nbd-module complexes.
\end{proposition}

\begin{figure}[ht]
  \centering
  \begin{diagram}[small,LaTeXeqno]
    \label{eq:diag}
    C^{+} \tensor_{\RR_{0}} \Rp 1 & \rTo[l>=4em]^{\iota - \mu} & C^{+}
    \tensor_{\RR_{0}} \Rp 0 \\
    \dTo<{\alpha \tensor \id} && \dTo<{\alpha \tensor \id} \\
    D \tensor_{\RR_{0}} \Rp 1 & \rTo^{\psi} & D \tensor_{\RR_{0}} \Rp
    0
  \end{diagram}
  \caption{The \textsc{Mather} trick square}
  \label{fig:Mather}
\end{figure}

\begin{proof}
  By construction, $J$ is a homotopy from
  $(\alpha \tensor \id) \circ (\iota - \mu)$
  to~$\psi \circ (\alpha \tensor \id)$.
  Hence we obtain a chain map of the mapping cones of the horizontal
  maps in the diagram,%
  \begin{displaymath}
    \alpha_{*} = \begin{pmatrix} \alpha \tensor \id & 0 \\ J & \alpha
      \tensor \id \end{pmatrix} \colon
    \htorusp (C^{+}) = \cone (\iota - \mu) \rTo \cone (\psi) \ ;
  \end{displaymath}
  this map is a quasi-isomorphism since $\alpha$ is a homotopy
  equivalence (so the induced map on homology will be represented by a
  lower triangular matrix with isomorphisms on the main diagonal).
\end{proof}

\begin{corollary}
  \label{cor:fd-implies-fd}
  If $C^{+}$ is an $\RR_{0}$\nbd-finitely dominated bounded below
  chain complex of projective $\Rpp$\nbd-modules, then $C^{+}$ is
  $\Rpp$\nbd-finitely dominated, that is, $C^{+}$~is homotopy
  equivalent to a bounded complex of \textit{finitely generated}
  projective $\Rpp$\nbd-modules.
\end{corollary}

\begin{proof}
  As $C^{+}$ is $\RR_{0}$\nbd-finitely dominated we can choose an
  $\RR_{0}$\nbd-linear chain homotopy equivalence
  $\alpha \colon C^{+} \rTo D$ from~$C^{+}$ to a bounded complex~$D$
  of finitely generated projective $\RR_{0}$\nbd-modules. By
  Corollary~\ref{cor:C+-is-half-torus} and
  Proposition~\ref{prop:Mather_trick} there are quasi-isomorphisms
  \begin{equation}
    \label{eq:C+_is_cone_psi}
    C^{+} \lTo^{\simeq} \htorusp (C^{+}) \rTo^{\simeq} \cone (\psi)
  \end{equation}
  with
  $\psi \colon D \tensor_{\RR_{0}} \Rp 1 \rTo D \tensor_{\RR_{0}} \Rp
  0$
  as defined in Proposition~\ref{prop:Mather_trick} a map of bounded
  complexes of finitely generated projective $\Rpp$\nbd-modules. It
  follows that $C^{+}$ is quasi-isomorphic, hence homotopy equivalent,
  to a bounded complex of finitely generated projective
  $\Rpp$\nbd-modules as claimed.
\end{proof}

\section{Finite domination and homotopy finiteness}

It is an interesting question whether in the situation of
Corollary~\ref{cor:fd-implies-fd} the complex~$C^{+}$ is actually
$\Rpp$\nbd-\textit{homotopy finite}, that is, homotopy equivalent to a
bounded complex of finitely generated \textit{free}
$\Rpp$\nbd-modules. In general this turns out to be false; however,
when working with~$\R$ instead of~$\Rpp$ the analogous question has a
positive answer. --- As before, let $\RR = \R$ be a strongly
$\bZ$\nbd-graded ring.

\begin{proposition}
  Suppose that $C$ is a bounded complex of finitely generated
  projective $\R$\nbd-modules. If $C$ is $\RR_{0}$\nbd-finitely
  dominated, then $C$~is $\R$\nbd-homotopy finite, \ie, $C$~is
  homotopy equivalent to a bounded complex of finitely generated free
  $\R$\nbd-modules.
\end{proposition}

\begin{proof}
  Let $D$ be a bounded chain complex of finitely generated projective
  $\RR_{0}$\nbd-modules chain homotopy equivalent to~$C$. Then by the
  \textsc{Mather} trick for algebraic tori
  \cite[Lemma~3.7]{MR3704930}, $C$~is homotopy equivalent, as an
  $\R$\nbd-module complex, to the mapping cone of a certain self map
  of the induced complex $D \tensor_{\RR_{0}} \R$. Hence the
  finiteness obstruction of~$C$ in $\tilde K_{0} \big( \R \big)$
  vanishes whence $C$~is $\R$\nbd-homotopy finite.
\end{proof}

The analogous statement holds over a \textit{polynomial}
ring~$\RR_{0}[t]$ with a central indeterminate~$t$: \textit{An
  $\RR_{0}$\nbd-finitely dominated, bounded $\RR_{0}[t]$\nbd-module
  complex~$C^{+}$ of finitely generated projective modules is
  $\RR_{0}[t]$\nbd-homotopy finite.}  For $C^{+} \simeq \cone(\psi)$
as in~\eqref{eq:C+_is_cone_psi} and~\eqref{eq:diag}, and the
finiteness obstruction of the mapping cone vanishes since
$t\RR_{0}[t] \iso \RR_{0}[t]$. --- In general, however, this line of
reasoning fails when working over~$\Rpp$.

\begin{example}
  \textit{There exist a strongly $\bZ$\nbd-graded ring $\RR = \R$
    together with a bounded complex~$C^{+}$ of finitely generated
    projective $\Rpp$\nbd-modules such that $C^{+}$ is
    $\RR_{0}$\nbd-finitely dominated but not homotopy equivalent to a
    bounded complex of finitely generated free $\Rpp$\nbd-modules.}
  Specifically\footnote{I am indebted to R.~Hazrat for bringing this
    example to my attention.}, let $K$ be a field and let $\RR = \R$
  be the \textsc{Leavitt} $K$\nbd-algebra of type~$(1,1)$, that is,
  the (non-commutative) $K$\nbd-algebra on generators $A$, $B$, $C$,
  $D$ subject to the relations
  \begin{displaymath}
     AB+CD=1 \ , \quad BA=DC=1 \ , \quad BC=DA=0 \ ;
  \end{displaymath}
  we declare that $A$ and~$C$ have degree~$-1$, while $B$ and~$D$ are
  given degree~$1$. This is a $\bZ$\nbd-graded ring since all
  relations are homogeneous of degree~$0$. It is strongly graded by
  Corollary~\ref{cor:pou-strong} as the relations $AB+CD=1$ and $BA=1$
  provide partitions of unity of types~$(-1,1)$ and~$(1,-1)$,
  respectively. It is known that $\RR_{0}$ can be identified with an
  increasing union $\bigcup_{n \geq 0} \mathbf{Mat}_{2^n}(K)$ of
  matrix algebras, using the block-diagonal embeddings
  $x \mapsto \left( \begin{smallmatrix} x&0 \\ 0 & x \end{smallmatrix}
  \right)$.
  It follows that $\RR_{0}$ has IBN, and since the projection map
  $\Rpp = \bigoplus_{k \geq 0} \RR_{k} \rTo \RR_{0}$ is a ring
  homomorphism, so does $\Rpp$. --- The $\Rpp$\nbd-module $Q = \Rp{1}$
  is finitely generated projective by
  Corollary~\ref{cor:strong-means-fgp}~(3), and the map
  \begin{equation}
    \label{eq:map_iso}
    \Rpp \rTo^{\iso} Q \oplus Q \ , \quad r \mapsto (Br,\, Dr)
  \end{equation}
  is an isomorphism of $\Rpp$\nbd-modules with inverse
  $(x,y) \mapsto Ax + Cy$. In addition, $Q$~is not stably free: if
  $Q \oplus \Rpp^{m} \iso \Rpp^{n}$, then by~\eqref{eq:map_iso} also
  \begin{displaymath}
    \Rpp^{2n} \iso \big( Q \oplus \Rpp^{m} \big) \oplus \big( Q \oplus
    \Rpp^{m} \big) \iso \Rpp^{2m+1} \ ;
  \end{displaymath}
  as $\Rpp$ has IBN, the inequality $2n \neq 2m+1$ renders this
  impossible. The class of~$Q$ in $\tilde K_{0} \big( \Rpp \big)$ is
  thus non-zero, and has in fact order~$2$ in view of the
  isomorphism~\eqref{eq:map_iso} . Thus the inclusion map
  $Q \rTo \Rpp$, considered as a chain complex~$C^{+}$, is an example
  of a bounded complex of finitely generated projective
  $\Rpp$\nbd-modules not homotopy equivalent to a bounded complex of
  finitely generated free $\Rpp$\nbd-modules. On the other hand, the
  complex~$C^{+}$ is certainly $\RR_{0}$\nbd-finitely dominated since
  the cokernel of the inclusion map $Q = \Rp {1} \rTo \Rpp$ is
  isomorphic to~$\RR_{0}$.
\end{example}

\section{$R_{0}$-finite domination of $\Rpp$-module complexes}

We now develop a homological criterion to detect if a chain
complex~$C^{+}$ of $\Rpp$\nbd-modules is $\RR_{0}$\nbd-finitely
dominated when considered as a complex of $\RR_{0}$\nbd-modules
\textit{via} restriction of scalars. This happens if and only if
$C^{+}$~has trivial \textsc{Novikov} homology in the sense that the
induced chain complex $C^{+} \tensor_{\Rpp} \Novm$ is acyclic.

\begin{theorem}
  \label{thm:R0-fd-R+}
  Let $\RR = \R$ be a strongly $\bZ$\nbd-graded ring, and let $C^{+}$ be
  a bounded chain complex of finitely generated projective
  $\Rpp$-modules. The following statements are equivalent:
  \begin{enumerate}
  \item The complex $C^{+}$ is $\RR_{0}$-finitely dominated.
  \item The complex $C^{+} \tensor_{\Rpp} \Novm$ is contractible (\ie,
    $C^{+}$ has trivial \textsc{Novikov} homology).
  \end{enumerate}
\end{theorem}

\begin{proof}
  (1)~$\Rightarrow$~(2): As $C^{+}$ is $\RR_{0}$\nbd-finitely
  dominated, we find a bounded complex~$D$ of finitely generated
  projective $\RR_{0}$\nbd-modules, and mutually inverse
  $\RR_{0}$\nbd-linear chain homotopy equivalences
  $\alpha \colon C^{+} \rTo D$ and $\beta \colon D \rTo C^{+}$. Let
  $\psi$ be as in Proposition~\ref{prop:Mather_trick}; together with
  Corollary~\ref{cor:C+-is-half-torus}, the \textsc{Mather} trick
  asserts that the $\Rpp$\nbd-module complexes $C^{+}$ and
  $\cone(\psi)$ are quasi-iso\-mor\-phic and thus are chain homotopy
  equivalent (as both are bounded below and consist of projective
  $\Rpp$\nbd-modules). Thus $C^{+} \tensor_{\Rpp} \Novm$ is homotopy
  equivalent to $\cone(\psi) \tensor_{\Rpp} \Novm$. The latter complex
  in turn is isomorphic to the mapping cone of the chain map
  \begin{displaymath} 
    D \tensor_{\RR_{0}} \Novm \rTo D \tensor_{\RR_{0}} \Novm
  \end{displaymath}
  sending the element $x \tensor \sum_{i \leq k} r_{i} t^{i}$ to
  \begin{displaymath}
    \alpha\beta(x) \tensor \sum_{j \leq k} r_{j} t^{j} - \sum_{j}
    \alpha \big( \beta(x) \alpha_{j}^{(1)} \big) \tensor \sum_{j \leq
      k} \beta_{j}^{(-1)} r_{j} t^{j-1} \ ,
  \end{displaymath}
  where we write elements of $\Novm$ as formal \textsc{Laurent} series
  in~$t\inv$; note that in the target of the map,
  $\beta_{j}^{(-1)} r_{j}$ is the coefficient of~$t^{j-1}$ as
  $\beta_{j}^{(-1)}$ has degree~$-1$.

  As $D$ consists of finitely presented $\RR_{0}$\nbd-modules, we can
  identify both the tensor products
  \begin{gather*}
    D \tensor_{\RR_{0}} \Rpp \tensor_{\Rpp} \Novm = D
    \tensor_{\RR_{0}} \Novm \\ %
    \intertext{and} %
    D \tensor_{\RR_{0}} t^{1} \Rpp \tensor_{\Rpp} \Novm = D
    \tensor_{\RR_{0}} \Novm %
  \end{gather*}
  with the twisted right-truncated power of~$D$
  \cite[Proposition~3.13]{MR3704930}, that is,
  \begin{displaymath}
    D \tensor_{\RR_{0}} \Novm \iso \prod_{n \leq 0} \big( D
    \tensor_{\RR_{0}} \RR_{n} \big) \,\oplus\, \bigoplus_{n > 0} \big(
    D \tensor_{\RR_{0}} \RR_{n} \big) \ .
  \end{displaymath}
  Thus we rewrite $\cone(\psi) \tensor_{\Rpp} \Novm$ as the
  right-truncated totalisation \cite[Definition~1.1]{MR2978500} of a
  double complex
  \begin{displaymath}
    Z_{p,q} = (D_{p+q+1} \tensor_{\RR_{0}} \RR_{p}) \, \oplus \,
    (D_{p+q} \tensor_{\RR_{0}} \RR_{p}) %
  \end{displaymath}
  with vertical differential $d^v \colon Z_{p,q} \rTo Z_{p, q-1}$ and
  horizontal differential $d^h \colon Z_{p,q} \rTo Z_{p-1, q}$ given
  by the formul\ae{}
  \begin{gather*}
    d^v (x \tensor a, y \tensor b ) = \big( -d(x) \tensor a, \alpha
    \beta (x)
    \tensor a + d(y) \tensor b \big) \ , \\
    d^h (x \tensor a, y \tensor b) = \Big( -\sum \alpha \big( \beta(y)
    \alpha_{j}^{(1)} \big) \tensor \beta_{j}^{(-1)} b, 0 \Big) \ .
  \end{gather*}
  The symbol ``$d$'', without any decorations, refers to the
  differential of the chain complex~$D$. The columns are acyclic since
  $Z_{p, *}$ is a shift suspension of
  $\cone(\alpha\beta) \tensor_{\RR_{0}} \RR_{p}$ and the chain map
  $\alpha\beta$ is homotopic to an identity map. It follows that
  $C^{+} \tensor_{\Rpp} \Novm \simeq \cone(\psi) \tensor_{\Rpp} \Novm$
  is acyclic \cite[Proposition~1.2]{MR2978500}, and hence
  contractible.

  \medbreak

  (2)~$\Rightarrow$~(1): As $C^{+}$ consists of finitely generated
  projective $\Rpp$\nbd-mod\-ules, there exists another bounded
  complex~$B^{+}$ with zero differentials, consisting of finitely
  generated projective $\Rpp$\nbd-modules, such that
  $A^{+} = B^{+} \oplus C^{+}$ is a bounded complex of finitely
  generated free $\Rpp$\nbd-modules. We equip $A^{+}_{k}$ with a basis
  with $r_{k}$ elements, and identify $A^{+}_{k}$ with~$\Rpp^{r_{k}}$
  henceforth. The differential $d_{k} \colon A^{+}_{k} \rTo
  A^{+}_{k-1}$ is thus represented by a matrix~$D_{k}$ with entries
  in~$\Rpp$.

  Suppose, for ease of notation, that $C^{+}$ is concentrated in chain
  levels between~$0$ and~$m$. We can choose integers
  \begin{displaymath}
    d_{m} \leq d_{m-1} \leq \ldots \leq d_{0} = -1
  \end{displaymath}
  so that $D_{k}$ defines a map
  $d^{-}_{k} \colon t^{d_{k}} \Psm^{r_{k}} \rTo t^{d_{k-1}}
  \Psm^{r_{k-1}}$;
  we only need to ensure that no entry of~$D_{k}$ has a non-zero
  homogeneous component of degree exceeding $d_{k-1}-d_{k}$. We let
  $S$ denote the chain complex thus defined, with $S_{k} = t^{d_{k}}
  \Psm^{r_{k}}$ and differentials~$D_{k}$. Similarly, we let $N$ denote
  the chain complex with $N_{k} = \Novm^{r_{k}}$ and
  differentials~$D_{k}$. Note that $S$ is a subcomplex of~$N$.

  For any $d \leq -1$ there is a short exact sequence of
  $\RR_{0}$\nbd-modules
  \begin{equation}
    \label{seq:1}
    0 \rTo t^{d} \Psm \oplus \Rpp
    \rTo[l>=3em]^{(-1 \ 1)} \Novm \rTo \bigoplus_{j=d+1}^{-1} \RR_{j}
    \rTo 0 \ %
  \end{equation}
  with last term a finitely generated projective
  $\RR_{0}$\nbd-module by Corollary~\ref{cor:strong-means-fgp}, as
  $\R$~is strongly graded. It follows that there is a short exact
  sequence of $\RR_{0}$\nbd-module complexes
  \begin{equation}
    \label{seq:2}
    0 \rTo S \oplus A^{+} \rTo[l>=4em]^{(-1 \ 1)}_{\beta} N
    \rTo P \rTo 0 
  \end{equation}
  with $P$ a bounded complex of finitely generated projective
  $\RR_{0}$\nbd-modules. In chain degree~$k$ this sequence is actually
  just the $r_{k}$\nbd-fold direct sum of~\eqref{seq:1} with itself,
  for $d = d_{k}$.

  From the sequence~\eqref{seq:2} we infer that the map from the
  mapping cone of~$\beta$ to~$P$ is a quasi-isomorphism. Now recall
  $A^{+} = B^{+} \oplus C^{+}$ and observe the consequent splitting
  \begin{equation}
    \label{eq:2}
    N  = A^{+} \tensor_{\Rpp} \Novm = B^{+} \tensor_{\Rpp} \Novm \ %
    \oplus \ C^{+} \tensor_{\Rpp} \Novm \ . 
  \end{equation}
  By hypothesis $C^{+} \tensor_{\Rpp} \Novm$ is contractible; thus $N$
  is quasi-isomorphic, \textit{via} the projection map, to
  $B^{+} \tensor_{\Rpp} \Novm$. As taking mapping cones is homotopy
  invariant, we can replace~$N$ by the latter complex and conclude
  that $P$ is quasi-isomorphic to the mapping cone of the map
  \begin{displaymath}
    \gamma \colon S \oplus A^{+} = S \oplus B^{+} \oplus C^{+} %
    \rTo[l>=4em]^{(-1 \ 1 \ 0)} B^{+} \tensor_{\Rpp} \Novm \ .
  \end{displaymath}
  As $\gamma$ is the zero map on the $C^{+}$\nbd-summand, the mapping
  cone of~$\gamma$ contains~$C^{+}$ as a direct summand.  Hence in the
  derived category of the ring~$\RR_{0}$, the complex $C^{+}$ is a
  retract of~$P$. Since both complexes are bounded and consist of
  projective $\RR_{0}$\nbd-modules, we conclude that $C^{+}$ is a
  retract up to homotopy of~$P$ whence $C^{+}$ is
  $\RR_{0}$\nbd-finitely dominated as claimed.
\end{proof}

\section{$\Rpp$-\textsc{Fredholm} matrices}
\label{sec:rpp-Fredhold-matrices}

Let $\RR = \R$ be a $\bZ$-graded ring, and let $A^{+}$ be a non-zero
square matrix of size~$k$ with entries in~$\R$. For suitable
$m \in \bZ$, multiplication by~$A^{+}$ defines an $\Rpp$\nbd-module
homomorphism
\begin{equation*}
  A^{+} = \mu (A^{+}, m) \colon \Rpp^{k} \rTo \big( t^{-m} \Rpp
  \big)^{k} \ , \quad x \mapsto A^{+} \cdot x \ ;
\end{equation*}
``suitable'' means, in fact, that $-m$ is not larger than the minimal
degree of non-zero homogeneous components of entries
of~$A^{+}$. Suppose now that in addition to such~$m$ we fix an integer
$n>m$ so that the map $\mu(A^{+}, n)$ is defined as well.

\begin{lemma}
  \label{lem:m-n-no-matter}
  There is an isomorphism of $\RR_{0}$\nbd-modules
  \begin{displaymath}
    \coker \mu(A^{+}, n) \iso \coker \mu(A^{+}, m) \oplus
    \bigoplus_{k=-n}^{-m-1} \RR_{k} \ .
  \end{displaymath}
\end{lemma}

\begin{proof}
  The direct sum of the exact sequence of $\RR_{0}$\nbd-modules
  \begin{displaymath}
    \Rpp^{k} \rTo[l>=4em]^{\mu (A^{+}, m)} %
    \big( t^{-m} \Rpp \big)^{k} \rTo %
    \coker \mu (A^{+}, m) \rTo 0
  \end{displaymath}
  with the exact sequence
  \begin{displaymath}
    0 \rTo \bigoplus_{k=-n}^{-m-1} \RR_{k} \rTo[l>=3em]^{=}
    \bigoplus_{k=-n}^{-m-1} \RR_{k}  \rTo 0
  \end{displaymath}
  yields a new exact sequence, which is precisely the sequence
  \begin{displaymath}
    \Rpp^{k} \rTo[l>=4em]^{\mu (A^{+}, n)} %
    \big( t^{-n} \Rpp \big)^{k} \rTo %
    \coker \mu (A^{+}, m) \oplus \bigoplus_{k=-n}^{-m-1} \RR_{k} \rTo
    0 \ .
  \end{displaymath}
  Hence
  $\coker \mu (A^{+}, n) \iso \coker \mu (A^{+}, m) \oplus
  \bigoplus_{k=-n}^{-m-1} \RR_{k}$ as $\RR_{0}$\nbd-modules.
\end{proof}

\begin{corollary}
  \label{cor:m-n-no-matter}
  Suppose that $\RR = \R$ is \emph{strongly} $\bZ$\nbd-graded. In the
  situation of Lemma~\ref{lem:m-n-no-matter}, the module
  $\coker \mu(A^{+}, n)$ is a finitely generated projective
  $\RR_{0}$\nbd-module if and only if $\coker \mu(A^{+}, m)$ is.
\end{corollary}

\begin{proof}
  This is a consequence of Corollary~\ref{cor:strong-means-fgp}~(1)
  and Lemma~\ref{lem:m-n-no-matter}.
\end{proof}

\begin{proposition}
  \label{prop:Fredholm_matrices+}
  Suppose that $\RR = \R$ is a strongly $\bZ$\nbd-graded ring. Let
  $A^{+}$ be a $k \times k$-matrix with entries in~$\R$, and let
  $m \in \bZ$ be ``suitable'' in the sense that multiplication
  by~$A^{+}$ yields a map of finitely generated projective
  $\Rpp$\nbd-modules
  $A^{+} = \mu(A^{+}, m) \colon \Rpp^{k} \rTo \Rp{-m}^{k}$,
  $x \mapsto A^{+} \cdot x$ (see discussion above) which we may
  consider as a chain complex
  concentrated in chain degrees~1 and~$0$. The following statements
  are equivalent:\goodbreak
  \begin{enumerate}
  \item The chain complex $A^{+}$ is \fd.
  \item The induced chain complex $A^{+} \tensor_{\Rpp} \Novm$ is
    contractible.
  \item The map~$A^{+}$ is invertible over~$\Novm$, that is, the
    map
    \begin{displaymath}
      \Novm^{k} \rTo \Novm^{k} \ , \quad x \mapsto A^{+} \cdot x
    \end{displaymath}
    is an isomorphism.
  \item The matrix~$A^{+}$ is invertible in the ring of all square
    matrices of size~$k$ with entries in~$\Novm$.
  \item The map $\mu(A^{+},m)$ is injective, and $\coker \mu(A^{+},m)$
    is a finitely generated projective $R_{0}$-module.
  \end{enumerate}
  Moreover, the validity of these statements does not depend on the
  specific choice of a suitable $m \in \bZ$.
\end{proposition}

\begin{definition}
  \label{def:Omega_+}
  A square matrix matrix with entries in~$\R$ satisfying one (and
  hence all) of the conditions listed in
  Proposition~\ref{prop:Fredholm_matrices+} is called an
  $\Rpp$\nbd-\textsc{Fredholm} \textit{matrix}. The set of all
  $\Rpp$\nbd-\textsc{Fredholm} matrices (of arbitrary finite size) is
  denoted by the symbol~$\Omega_{+}$.
\end{definition}

\begin{proof}[Proof of Proposition~\ref{prop:Fredholm_matrices+}]
  Condition~(5) is insensitive to the precise value of the
  suitable integer~$m$, in view of Corollary~\ref{cor:m-n-no-matter}.

  The equivalence of conditions~(1) and~(2) is
  Theorem~\ref{thm:R0-fd-R+} above. Statements~(3) and~(4) are
  trivially equivalent.

  By Lemma~\ref{lem:mult_iso}, the multiplication map
  \begin{displaymath}
    t^{-m} \Rpp \tensor_{\Rpp} \R \rTo \R \ , \quad x \tensor y
    \mapsto xy
  \end{displaymath}
  is an isomorphism of $\R$-modules. It follows that there is a chain
  of isomorphisms
  \begin{multline*}
    t^{-m} \Rpp \tensor_{\Rpp} \Novm \iso t^{-m} \Rpp \tensor_{\Rpp}
    \R \tensor_{\R} \Novm \\ %
    \iso \R \tensor_{\R} \Novm \iso \Novm
  \end{multline*}
  with composition the multiplication map. In view of this,
  statements~(2) and~(3) are equivalent.

  If (5) holds then the chain complex~$A^{+}$ is $R_{0}$\nbd-homotopy
  equivalent to the module $\coker \mu(A^{+},m)$, considered as a
  chain complex concentrated in degree~$0$, which shows that~(1) is
  satisfied in this case.

  Suppose finally that (3) holds; we will show that (5) is valid as
  well. We infer from the commutative square
  \begin{diagram}[small]
    \Rpp^{k} & \rTo[l>=5em]^{\mu(A^{+},m)} & \Rp{-m}^{k} \\
    \dTo<{\subset} && \dTo<{\subset} \\
    \Novm^{k} & \rTo[l>=3em]^{A^{+}}_{\iso} & \Novm^{k}
  \end{diagram}
  that the map~$\mu (A^{+},m)$ must be injective. Thus it remains to
  verify that $\coker \mu(A^{+},m)$ is a finitely generated projective
  $\RR_{0}$-module. Assuming $m \geq 1$, as we may in view of
  Corollary~\ref{cor:m-n-no-matter}, we can embed~$\mu(A^{+},m)$ into
  a commutative diagram of $\RR_{0}$\nbd-modules
  \begin{diagram}[righteqno,LaTeXeqno,small]
    \label{diag:large}
    \ignore( t\inv \RR \powers{t\inv} \ignore)^{k} & \rTo^{\subset} &
    \ignore( \RR \nov{t\inv} \ignore)^{k} & \lTo^{\supset} & \ignore( \Rpp \ignore)^{k} \\ %
    \dTo<{A^{+}} && \dTo<{A^{+}} &&
    \dTo<{A^{+}}>{\mu(A^{+},m)} \\ %
    \ignore( t^{q} \RR \powers{t\inv}\ignore)^{k} &
    \rTo[l>=3em]^{\subset} & \ignore( \RR \nov{t\inv} \ignore)^{k} &
    \lTo[l>=3em]^{\supset} & \ignore( t^{-m} \Rpp \ignore)^{k}
  \end{diagram}
  where $q \geq 0$ is sufficiently large; it is sufficient that $q$
  exceeds the maximal degree of any non-zero homogeneous component of
  the entries of~$A^{+}$.  Now in any \textsc{abel}ian category, a
  diagram $\mathcal{D} = \big( X \rTo^{\xi} Y \lTo^{\zeta} Z \big)$
  gives rise to an exact sequence, natural in~$\mathcal{D}$, of the
  form
  \begin{displaymath}
    0 \rTo \ker(\xi-\zeta) \rTo X \oplus Z \rTo[l>=3em]^{\xi - \zeta}
    Y  \rTo \coker(\xi-\zeta) \rTo 0 \ .
  \end{displaymath}
  We apply this to the rows of diagram~\eqref{diag:large} above,
  noting the the $\coker$ term is trivial in both cases (since
  $q,m \geq 0$). The kernel, on the other hand, is trivial in case of
  the top row, and is the finitely generated projective
  $\RR_{0}$\nbd-module $P = \bigoplus_{-m}^{q} \RR_{j}$ for the bottom
  row. We arrive at the following commutative diagram with exact rows:
  \begin{diagram}[small]
    && 0 & \rTo & t\inv \RR\nov{t\inv}^{k} \oplus \Rpp^{k} & \rTo & \RR
    \nov{t\inv} & \rTo & 0 \\
    && \dTo && \dTo<{A^{+} \oplus A^{+}} && \dTo<{A^{+}} \\
    0 & \rTo & P & \rTo & t^{q} \RR\nov{t\inv}^{k} \oplus
    t^{-m}\Rpp^{k} & \rTo & \RR \nov{t\inv} & \rTo & 0
  \end{diagram}
  As the right-hand vertical map is an isomorphism by hypothesis~(3),
  the \textsc{Snake} lemma yields an isomorphism of~$P$ with the
  cokernel of the middle vertical map, which contains
  $\coker \mu(A^{+},m) \colon \Rpp^{k} \rTo t^{-m} \Rpp^{k}$ as a
  direct summand. This shows that $\coker \mu(A^{+},m)$ is a finitely
  generated projective $\RR_{0}$\nbd-module as desired.
\end{proof}

\section{The \textsc{Fredholm} localisations $\Omega_{+}\inv \Rpp$ and~$\Omega_{+}\inv \R$}

We now turn our attention to the non-commutative localisations
\begin{displaymath}
  \alpha \colon \Rpp \rTo \Omega_{+}\inv \Rpp %
  \quad \text{and} \quad %
  \gamma \colon \R \rTo \Omega_{+}\inv \R
\end{displaymath}
where $\Omega_{+}$ denotes the set of $\Rpp$\nbd-\textsc{Fredholm}
matrices as in Definition~\ref{def:Omega_+}. 
To be precise, we define
$\alpha = \lambda_{\Omega_{+}} \colon \Rpp \rTo \Omega_{+}\inv \Rpp$
as the the non-commutative localisation inverting all the maps
\begin{equation}
  \label{eq:invert_maps_polynomial}
  \mu (A^{+},m) \colon \Rpp^{k} \rTo \Rp {-m}^{k}
\end{equation}
of finitely generated projective $\Rpp$\nbd-modules, where $k \geq 1$
is arbitrary, $A^{+} \in \Omega_{+}$ has size~$k$, and $m \in \bZ$ is
suitable in the sense of \S\ref{sec:rpp-Fredhold-matrices}. As
$A^{+}$ satisfies property~(4) of
Proposition~\ref{prop:Fredholm_matrices+}, the universal property of
non-commutative localisation yields a factorisation
\begin{equation}
  \label{eq:factorisation}
  \Rpp \rTo^{\alpha} \Omega_{+}\inv \Rpp \rTo \RR \nov{t\inv}
\end{equation}
of the inclusion map; in particular, $\alpha$~is injective. ---
Similarly, we define
$\gamma = \lambda_{\Omega_{+}} \colon \R \rTo \Omega_{+}\inv \R$ as
the the non-commutative localisation inverting all the maps
\begin{equation}
  \label{eq:invert_maps_Laurent}
  A^{+} \colon \R^{k} \rTo \R^{k}
\end{equation}
of finitely generated free $\R$\nbd-modules, where $k \geq 1$ is
arbitrary and $A^{+} \in \Omega_{+}$ has size~$k$. As $A^{+}$
satisfies property~(4) of Proposition~\ref{prop:Fredholm_matrices+},
the universal property of non-commutative localisation yields a
factorisation
\begin{equation}
  \label{eq:factorisation}
  \R \rTo^{\gamma} \Omega_{+}\inv \R \rTo \RR \nov{t\inv}
\end{equation}
of the inclusion map; in particular, $\gamma$~is injective.

\medbreak

Applying the functor $\,\mbox{-}\, \tensor_{\Rpp} \R$ to a map as
in~\eqref{eq:invert_maps_polynomial} yields a map as
in~\eqref{eq:invert_maps_Laurent}, by Lemma~\ref{lem:mult_iso}. Thus
$\gamma|_{\Rpp} \colon \Rpp \rTo \Omega_{+}\inv \R$ inverts all the
maps~\eqref{eq:invert_maps_polynomial} and factorises through a ring
homomorphism
$\delta \colon \Omega_{+}\inv \Rpp \rTo \Omega_{+}\inv \R$. That is,
the maps $\alpha$ and~$\gamma$ fit into the commutative square diagram
of Fig.~\ref{fig:po_Fredholm} which, by
Proposition~\ref{prop:pushout_by_loc}, is a pushout square in the
category of unital rings.

\begin{figure}[ht]
  \centering
\begin{diagram}[LaTeXeqno,small]
    \label{eq:factorise_Rpp_Novm}
    \Rpp & \rTo[l>=2em]^{\alpha} & \Omega_{+}\inv \Rpp \\
    \dTo<{\subset}>{\beta} & \po &
    \dTo<{\delta} \\
    \R & \rTo^{\gamma} & \Omega_{+}\inv \R
  \end{diagram}
  \caption{Pushout square of \textsc{Fredholm} localisations}
  \label{fig:po_Fredholm}
\end{figure}

\begin{theorem}
  \label{thm:non-com-fd+}
  Let $\RR = \R$ be a strongly $\bZ$\nbd-graded ring, and let $C^{+}$
  be a bounded chain complex of finitely generated free
  $\Rpp$-modules. The following statements are equivalent:
  \begin{enumerate}
  \item The chain complex~$C^{+}$ is \fd.
  \item The induced chain complex
    $C^{+} \tensor_{\Rpp} \Omega_{+}\inv \Rpp$ is contractible.
  \item The induced chain complex
    $C^{+} \tensor_{\Rpp} \Omega_{+}\inv \R$ is contractible.
  \end{enumerate}
\end{theorem}

\begin{proof}
  (2) \impl~(3): Immediate from the factorisation
  \begin{displaymath}
    \Rpp \rTo^{\alpha} \Omega_{+}\inv \Rpp \rTo \Omega_{+}\inv \R
  \end{displaymath}
  of~$\gamma$, see~\eqref{eq:factorise_Rpp_Novm} in
  Fig.~\ref{fig:po_Fredholm}.

  (3) \impl~(1): Immediate from the
  factorisation~\eqref{eq:factorisation} and
  Theorem~\ref{thm:R0-fd-R+}.

  (1) \impl~(2): Suppose that $C^{+}$ is \fd. The differentials
  of~$C^{+}$ can be thought of as matrices $D^{+}_{n}$ with entries
  in~$\Rpp$, after equipping the chain modules of~$C^{+}$ with finite
  bases and thereby identifying $C^{+}_{n}$ with a finite direct sum
  $\bigoplus \Rpp = \big( \Rpp \big)^{r_{n}}$. The $n$th chain module
  of the induced complex $C^{+} \tensor_{\Rpp} \Novm$ is then
  identified with $\Novm^{r_{n}}$. The induced complex is contractible
  by Theorem~\ref{thm:R0-fd-R+}, so there are matrices
  $\sigma^{+}_{n}$ with entries in~$\Novm$ such that
  $D^{+}_{n+1} \cdot \sigma^{+}_{n} + \sigma^{+}_{n-1} \cdot
  D^{+}_{n}$ is a unit matrix of
  size~$r_{n}$. 
  We can truncate the entries of the
  matrices~$\sigma^{+}_{n}$ below at some suitable integer $m \ll 0$
  (not depending on~$n$) to obtain matrices
  $S^{+}_{n} = \tr_{m}(\sigma^{+}_{n})$ with entries in~$\R$ such that
  $E_{n} = D^{+}_{n+1} \cdot S^{+}_{n} + S^{+}_{n-1} \cdot D^{+}_{n}$
  is the sum of a unit matrix, and a matrix the non-zero entries of
  which have homogeneous components of strictly negative degree. Thus
  $E_{n}$~is invertible over~$\Novm$ so that $E_{n} \in \Omega_{+}$.

  We now define a new $\Rpp$\nbd-module chain complex~$B^{+}$ by
  setting $B^{+}_n = \big( \Rp m \big)^{r_{n}}$, with differentials
  given by the matrices~$D^{+}_{n}$. The inclusion map
  $\Rpp \rTo \Rp m$, itself an element of~$\Omega_{+}$, yields a chain
  map $g \colon C^{+} \rTo B^{+}$; the induced chain map
  \begin{displaymath}
    g \tensor \Omega_{+}\inv \Rpp \colon C^{+} \tensor_{\Rpp}
    \Omega_{+}\inv \Rpp \rTo B^{+} \tensor_{\Rpp} \Omega_{+}\inv \Rpp
  \end{displaymath}
  is an isomorphism. The matrices~$S^{+}_{n}$ define module
  homomorphisms $C^{+}_{n} \rTo B^{+}_{n+1}$ which constitute in fact
  an $\alpha$\nbd-proto-null homotopy since
  $E_{n} = D^{+}_{n+1} \cdot S^{+}_{n} + S^{+}_{n-1} \cdot D^{+}_{n}$
  is an element of~$\Omega_{+}$, as explained above. Here
  $\alpha \colon \Rpp \rTo \Omega_{+} \Rpp$ is the localisation map as
  in~\eqref{eq:factorise_Rpp_Novm}. It follows that
  $C^{+} \tensor_{\Rpp} \Omega_{+} \inv \Rpp$, the source of
  $g \tensor \Omega_{+}\inv \Rpp$, is contractible by
  Lemma~\ref{lem:proto}.
\end{proof}

\begin{theorem}[Universal property of~$\Omega_{+}\inv \Rpp$]
  \label{thm:universal-Rpp}
  Suppose that $\R$ is a strongly $\bZ$\nbd-graded ring. The
  localisation $\lambda \colon \Rpp \rTo \Omega_{+}\inv \Rpp$ is the
  universal $\Rpp$\nbd-ring making $\RR_{0}$\nbd-finitely dominated
  chain complexes contractible. That is, suppose that
  $f \colon \Rpp \rTo S$ is an $\Rpp$\nbd-ring such that for every
  bounded complex of finitely generated free
  $\Rpp$\nbd-modules~$C^{+}$ which is $\RR_{0}$\nbd-finitely
  dominated, the complex $C^{+} \tensor_{\Rpp} S$ is
  contractible. Then there is a factorisation
  $\Rpp \rTo^{\lambda} \Omega_{+}\inv \Rpp \rTo^{\eta} S$ of~$f$, with
  a uniquely determined ring homomorphism~$\eta$.
\end{theorem}

\begin{proof}
  It was shown in Theorem~\ref{thm:non-com-fd+} above that
  $\Omega_{+}\inv \Rpp$ makes $\RR_{0}$\nbd-finitely dominated chain
  complexes contractible. Thus it is enough to show that $f$ inverts
  the maps $\mu(A^{+}, m)$ of~\eqref{eq:invert_maps_polynomial} for
  any $A^{+} \in \Omega_{+}$ and any suitable $m \in \bZ$. By
  definition of~$\Omega_{+}$ the \textit{complex}~$\mu(A^{+},m)$ is
  $\RR_{0}$\nbd-finitely dominated so that, by hypothesis on~$f$, the
  complex $\mu(A^{+},m) \tensor_{\Rpp} S$ is contractible. This says
  precisely that $f$ inverts the map $\mu(A^{+},m)$.
\end{proof}

One can also show that \textit{the localisation
  $\lambda \colon \R \rTo \Omega_{+}\inv \R$ is the universal
  $\R$\nbd-ring making $\RR_{0}$\nbd-finitely dominated, bounded chain
  complexes of finitely generated free $\Rpp$\nbd-modules complexes
  contractible}.

\medbreak

We finish with proving that
$\delta \colon \Omega_{+}\inv \Rpp \rTo \Omega_{+}\inv \R$ is an
isomorphism if $\R$ contains a homogeneous unit of non-zero degree.

\begin{proposition}
  \label{prop:hom-unit}
  Suppose that $\RR = \R$ is a strongly $\bZ$\nbd-graded ring. Suppose
  there exists a homogeneous unit of positive degree in~$\R$. Then
  there is an injective ring homomorphism
  $\iota \colon \R \rTo \Omega_{+}\inv \Rpp$ with
  $\iota \beta = \alpha$, and
  $\delta \colon \Omega_{+}\inv \Rpp \rTo \Omega_{+}\inv \R$ is an
  isomorphism satisfying $\delta \iota = \gamma$.
\end{proposition}

\begin{proof}
  Let $u \in \RR_{d} \cap \R^{\times}$, with $d>0$. Then the
  $1 \times 1$\nbd-matrix $(u)$ is an $\Rpp$\nbd-\textsc{Fred\-holm}
  matrix since the cokernel of the map
  \begin{displaymath}
    \Rpp \rTo \Rpp \ , \quad r \mapsto ur
  \end{displaymath}
  is the finitely generated projective $\RR_{0}$\nbd-module
  $\bigoplus_{0}^{d-1} \RR_{j}$. The induced map
  \begin{displaymath}
    \Omega_{+}\inv \Rpp \rTo \Omega_{+}\inv \Rpp
  \end{displaymath}
  is given by multiplication with $\alpha(u) \in \Omega_{+}\inv \Rpp$.
  Since the induced map is an isomorphism, $\alpha(u)$ is invertible
  in~$\Omega_{+}\inv \Rpp$.

  Given any $x \in \R$ there exists $k \geq 0$ with $u^{k} x \in \Rpp$
  and thus $\alpha(u^{k} x) \in \Omega_{+}\inv \Rpp$; we define
  $\iota(x) = \alpha(u)^{-k} \cdot \alpha(u^{k}x) \in \Omega_{+}\inv
  \Rpp$. The element $\iota(x)$ does not depend on the choice of~$k$,
  for if $\ell > k$ we have
  \begin{multline*}
    \alpha(u)^{-\ell} \cdot \alpha(u^{\ell}x) = \alpha(u)^{-k}
    \alpha(u)^{-\ell+k} \cdot \alpha(u^{\ell-k} u^{k}x) =
    \alpha(u)^{-k} \cdot \alpha(u^{k}x) \ ,
  \end{multline*}
  since $u^{\ell-k} \in \Rpp$ and since $\alpha$ is a ring
  homomorphism. Note that $\iota(x) = \alpha(x)$ for $x \in \Rpp$, and
  that $\iota(u\inv) = \alpha(u)\inv$.

  Suppose that $x \in \R$ and $k,\ell \geq 0$ are such that
  $u^{k} x u^{-\ell} \in \Rpp$. Then
  $\alpha(u^{k} x u^{-\ell}) = \alpha(u^{k} x) \alpha(u)^{-\ell}$,
  since both sides equal~$\alpha(u^{k} x)$ after multiplication with
  $\alpha(u)^{\ell}$. Consequently, for $x,y \in \R$ and
  $k, \ell \geq 0$ with $u^{\ell}y,\ u^k x u^{-\ell} \in \Rpp$ we
  calculate
  \begin{align*}
    \iota(xy)                                & = \iota(x u^{-\ell}
                                               u^{\ell} y) \\ 
                                             & = \alpha(u)^{-k} \cdot
                                               \alpha (u^{k} x
                                               u^{-\ell} u^{\ell} y) \\
                                             & = \alpha(u)^{-k} \cdot
                                               \alpha (u^{k}x u^{-\ell}) \cdot \alpha(u^{\ell} y)                         \\
                                             & = \alpha(u)^{-k} \cdot \alpha (u^{k} x) \cdot \alpha(u)^{-\ell}
                                               \cdot \alpha(u^{\ell} y) = \iota(x) \cdot \iota(y) \ .
  \end{align*}
  Since $\iota$ is clearly additive, the map
  $\iota \colon \R \rTo \Omega_{+}\inv \Rpp$ is thus a ring
  homomorphism. Moreover, $\iota$ is injective as
  $\iota(x) = \alpha(u)^{-k} \cdot \alpha(u^{k}x)$ vanishes if and
  only if $\alpha(u^{k}x)$ vanishes. It follows from
  Lemmas~\ref{lem:incl_is_epi} and~\ref{lem:cat-thy} that the ring
  homomorphism~$\delta$ is an isomorphism and satisfies
  $\delta \iota = \gamma$.
\end{proof}


\end{document}